\documentclass[10pt]{amsart}

\usepackage{amsmath,amssymb,amsfonts,amsthm,amscd,shadow,fancybox,epic}
\usepackage{pstricks,pst-node,pst-plot}

\theoremstyle{plain}
\newtheorem{thm}[subsection]{Theorem}
\newtheorem{lem}[subsection]{Lemma}
\newtheorem{prop}[subsection]{Proposition}
\newtheorem{cor}[subsection]{Corollary}
\newtheorem{ex}[subsection]{Example}
\newtheorem{rem}[subsection]{Remark}

\newtheorem{defn}[subsection]{Definition}

\newrgbcolor{ncyn}{0 1 1}
\newrgbcolor{ngrn}{0 .8 0}

\def\equivpiece{ncyn}
\def\onetriangle{ngrn}

\newcommand{\ga}{\alpha}

\newcommand{\gD}{\Delta}
\newcommand{\gl}{\lambda}
\newcommand{\gk}{\kappa}
\newcommand{\gm}{\mu}
\newcommand{\gn}{\nu}
\newcommand{\gx}{\xi}
\newcommand{\go}{\omega}

\newcommand{\gP}{\Phi}
\newcommand{\gr}{\rho}
\newcommand{\gs}{\sigma}

\newcommand{\bd}{\partial}

\newcommand{\lmn}{{_{\gl,\gm}^\gn}}
\newcommand{\lmnp}{{_{\gl,\gm}^{\gn\,+}}}

\newcommand{\gmk}{{\gm/\gk}}

\newcommand{\cB}{\mathcal{B}}
\newcommand{\cH}{\mathcal{H}}
\newcommand{\cL}{\mathcal{L}}
\newcommand{\cP}{\mathcal{P}}
\newcommand{\cR}{\mathcal{R}}

\newcommand{\fa}{a}
\newcommand{\fb}{b}

\newcommand{\bZ}{\mathbb{Z}}
\newcommand{\bN}{\mathbb{N}}
\newcommand{\bC}{\mathbb{C}}

\newcommand{\sgn}{\mathop{\rm sgn}\nolimits}

\newcommand{\xy}{(x\,|\,y)}

\newcommand{\ol}{\overline}
\newcommand{\wt}{\widetilde}
\newcommand{\es}{\emptyset}
\newcommand{\wh}{\widehat}
\newcommand{\ph}{\phantom}
\def\ni{\noindent}

\newcommand{\vs}{\vspace{1em}}

\title[Equivariant Littlewood-Richardson Skew Tableaux]{Equivariant Littlewood-Richardson\\ Skew Tableaux}
\author{Victor Kreiman}
\address{Department of Mathematics, University of Georgia\\
Athens, GA 30602 USA}
\email{vkreiman@math.uga.edu}
\date{\today}

\begin{document}
\maketitle

\begin{abstract}
We give a positive equivariant Littlewood-Richardson rule also
discovered independently by Molev. Our proof generalizes a proof
by Stembridge of the ordinary Littlewood-Richardson rule. We
describe a weight-preserving bijection between our indexing
tableaux and the Knutson-Tao puzzles.
\end{abstract}

\setcounter{tocdepth}{1} \tableofcontents

\section{Introduction}\label{s.introduction}

In \cite{Mo-Sa}, Molev and Sagan introduced a rule in terms of
barred tableaux for computing the structure constants $c\lmn$ for
products of two factorial Schur functions. Knutson and Tao
\cite{Kn-Ta} realized that under a suitable specialization these
are the structure constants $C\lmn$ for products of two Schubert
classes in the equivariant cohomology ring of the Grassmannian.
Knutson and Tao \cite{Kn-Ta} also gave a new rule for computing
$C\lmn$, i.e., an equivariant Littlewood-Richardson rule, which is
manifestly positive in the sense of Graham \cite{Gr}. Their rule
was expressed in terms of puzzles, generalizations of
combinatorial objects first introduced by Knutson, Tao, and
Woodward \cite{Kn-Ta-Wo}.

We describe a new nonnegative equivariant Littlewood-Richardson
rule, expressed in terms of skew barred tableaux, which was also
discovered independently by Molev \cite{Mo1}. By nonnegative we
mean that all of the coefficients are either positive or zero;
restricting to the positive coefficients then yields a positive
rule. In our proof, we compute the structure constants $c\lmn$ (as
do both \cite{Mo-Sa} and \cite{Mo1}), and then determine the
structure constants $C\lmn$ by specialization (as does
\cite{Mo1}). Our strategy for deriving the structure constants
$c\lmn$ is to generalize a concise proof by Stembridge
\cite{Stem1} of the ordinary Littlewood-Richardson rule from Schur
functions to factorial Schur functions. This method in fact yields
a more general result, namely, a generalization of Zelevinsky's
extension of the Littlewood-Richardson rule \cite{Ze}.

We illustrate a weight-preserving bijction $\gP$ between the skew
barred tableaux indexing positive coefficients and the Knutson-Tao
puzzles, thus giving a new proof of Knutson and Tao's equivariant
Littlewood-Richardson rule, and also demonstrating that our
positive rule is really the same rule as Knutson and Tao's, just
expressed in terms of different combinatorial indexing sets. We
extend $\gP$ to a bijection from all skew barred tableaux indexing
nonnegative coefficients to the set of trapezoid puzzles, which
are generalizations of puzzles. Our representation of the
bijections generalizes Tao's `proof without words' \cite[Figure
11]{Va}, which gives a bijection between tableaux and puzzles in
the nonequivariant setting.

The results of this paper were presented at the AMS Sectional
Meeting, Santa Barbara, CA, April 2005, and the University of
Georgia Algebra Seminar, August 2006.

\section{Statement of Results}\label{s.results}

Let $\bN$ denote the set of nonnegative integers, and let $n\geq
d$ be fixed positive integers. For $m\in\bN$, define $m':=d+1-m$.
For $\gl=(\gl_1,\ldots,\gl_d)\in\bN^d$, define
$|\gl|=\gl_1+\cdots+\gl_d$. Denote by $\cP_d$ the set of all such
$\gl$ which are \textbf{partitions}, i.e., such that
$\gl_1\geq\cdots\geq \gl_d$, and by $\cP_{d,n}$ the set of all
such partitions for which $\gl_1\leq n-d$. Let
$\gl=(\gl_1,\ldots,\gl_d), \gm=(\gm_1,\ldots,\gm_d),
\gr=(d-1,d-2,\ldots,0)$, and $1=(1,\ldots,1)$ be fixed elements of
$\cP_d$. For any sequence $i=i_1,i_2,\ldots,i_t$,
$i_j\in\{1,\ldots,d\}$, define the \textbf{content of $i$} to be
$\go(i)=(\gx_1,\ldots,\gx_d)\in\bN^d$, where $\gx_k$ is the number
of $k$'s in the sequence.

\subsection{Defining the Structure Constants $c\lmn$ for Products of Factorial
 Schur Functions}
A \textbf{reverse Young diagram} is a right and bottom justified
array of boxes. To $\gm$ we associate the reverse Young diagram
whose bottom row has length $\gm_1$, next to bottom row has length
$\gm_2$, etc. We also denote this reverse Young diagram by $\gm$.
The columns of a reverse Young diagram are numbered from right to
left and the rows from bottom to top.

\begin{figure}[!h]
\begin{center}
\input{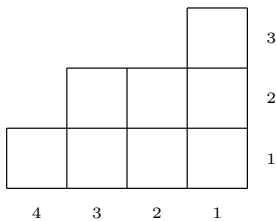}
\end{center}
\caption{\label{f.apic10_trev} The reverse Young diagram
$(4,3,1)$, with rows and columns numbered.}
\end{figure}

A \textbf{reverse tableau of shape $\gm$} is a filling of each box
of $\gm$ with an integer in $\{1,\ldots,d\}$ in such a way that
the entries weakly increase along any row from left to right and
strictly increase along any column from top to bottom. Let
$\cR(\gm)$ denote the set of all reverse tableaux of shape $\gm$.
Let $x_1,\ldots,x_d$ be a finite set of variables and
$(y_i)_{i\in\bN_{>0}}$ an infinite set of variables.  For
$R\in\cR(\gm)$, define
\begin{equation*}
\xy^R=\prod_{\fa\in R}\left(x_{\fa}-y_{a'+c(a)-r(a)}\right),
\end{equation*}
where for entry $a\in R$, $c(a)$ and $r(a)$ are the column and row
numbers of $a$ respectively. The \textbf{factorial Schur function}
is defined to be
\begin{equation*}
s_{\gm}\xy=\sum_{R\in \cR(\gm)}\xy^R.
\end{equation*}
Factorial Schur functions are special cases of Lascoux and
Sch\"utzenberger's double Schubert polynomials \cite{La-Sc1,
La-Sc2}. Various versions of factorial Schur functions and their
properties have been introduced and studied by \cite{Bi-Lo},
\cite{Ch-Lo}, \cite{Go-Gr}, \cite{La}, \cite{Mac1}, \cite{Mac2},
\cite{Mo1}, and \cite{Mo2} (see \cite{Mi}, \cite{Mo1}, and
\cite{Mo-Sa} for more discussion of these polynomials).

We check that our definition of factorial Schur function agrees
with the definition in \cite{Mo-Sa}, which is expressed in terms
of Young tableaux. Replacing each entry $a$ in a reverse tableau
$R$ by $a'$ and rotating the resulting tableau by $180$ degrees,
one obtains a Young tableau $T$. This operation defines a
bijection between reverse tableax of shape $\gm$ and Young
tableaux of shape $\gm$. The polynomials $\xy^T$, as defined in
\cite{Mo-Sa}, and $\xy^R$, as defined above, are related by a
fixed permutation on the indices of the $x_i$'s, namely the
involution $i\mapsto i'$. Thus the equivalence of the two
definitions follows from the fact that factorial Schur functions
are symmetric in the $x_i$'s. (Corollary \ref{c.alternant} also
establishes the equivalence of the two definitions.)

From the definition of $s_\gm\xy$, one sees that
\begin{equation*}
s_\gm\xy=s_\gm(x)+\text{ terms of lower degrees in the
}x_i\text{'s},
\end{equation*}
where $s_\gm(x)$ is the Schur function in $x_1,\ldots,x_d$. Since
the Schur functions form a $\bC$-basis for
$\bC[x_1,\ldots,x_d]^{S_d}$, the factorial Schur functions must
form a $\bC[y_i]$-basis for $\bC[y_i][x_1,\ldots,x_d]^{S_d}$. Thus
\begin{equation}\label{e.fact_lr_defn}
s_\gl\xy s_\gm\xy=\sum c_{\gl,\gm}^{\gn}s_{\gn}\xy,
\end{equation}
for some polynomials $c_{\gl,\gm}^{\gn}\in \bC[y_i]$, where the
summation is over all $\gn\in\cP_d$.

Also from the definition one sees that $s_\gm\xy$ is a homogeneous
polynomial of degree $|\gm|$. Therefore
$|\gl|+|\gm|-|\gn|=\deg(c\lmn)$. If $|\gl|+|\gm|-|\gn|=0$, then
$c\lmn\in\bC$ is the ordinary Littlewood-Richardson coefficient
(see \cite{Fu}, \cite{Li-Ri}, \cite{Sa}).

\subsection{Computing the Structure Constants $c\lmn$}

The \textbf{skew diagram $\gl*\gm$} is obtained by placing the
Young diagram $\gl$ above and to the right of the reverse Young
diagram $\gm$ (see Figure \ref{f.apic11_tyam}). A \textbf{skew
barred tableau $L$ of shape $\gl*\gm$} is a filling of each box of
the subdiagram $\gl$ of $\gl*\gm$ with an element of
$\{1,\ldots,d\}$ and each box of the subdiagram $\gm$ of $\gl*\gm$
with an element of $\{1,\ldots,d\}\cup\{\ol{1},\ldots,\ol{d}\}$,
in such a way that the values of the entries, without regard to
whether or not they are barred, weakly increase along any row from
left to right and strictly increase along any column from top to
bottom. The \textbf{unbarred column word of $L$}, denoted by
$L^u$, is the sequence of unbarred entries of $L$ beginning at the
top of the rightmost column, reading down, then moving to the top
of the next to rightmost column and reading down, etc (the barred
entries are just skipped over in this process).  We say that that
the unbarred column word of $L$ is \textbf{Yamanouchi} if, when
one writes down the word and stops at any point, one will have
written at least as many ones as twos, at least as many twos as
threes, $\ldots$, at least as many $(d-1)$'s as $d$'s. The
\textbf{unbarred content of $L$} is $\go(L^u)$, the content of the
unbarred column word.

\begin{defn} An \textbf{equivariant Littlewood-Richardson skew tableau}
is a skew barred tableau whose unbarred column word is Yamanouchi.
We denote the set of all equivariant Littlewood-Richardson skew
tableaux of shape $\gl*\gm$ and unbarred content $\gn$ by
$\cL\cR\lmn$.
\end{defn}
\ni We remark that this definition forces the $i$-th row of $\gl$
to consist of $\gl_i$ unbarred $i$'s.

\begin{figure}[!h]
\begin{center}
\input{apic11_tyam}
\end{center}
\caption{\label{f.apic11_tyam} An equivariant
Littlewood-Richardson skew tableau of shape $(2,1,1)*(4,3,1)$ and
unbarred content $(3,3,2,1)$. The unbarred column word,
$1,1,2,3,2,4,3,1,2$, is Yamanouchi, as required.}
\end{figure}

For $L$ a skew barred tableau and $a\in L$, denote by $L^u_{<a}$
the portion of the unbarred column word of $L$ which comes before
reaching $a$ when reading entries from $L$. Define
\begin{equation}\label{e.c_L}
c_L=\prod_{\fa\in L\atop \fa\text{
barred}}\left(y_{|\fa|'+\go(L^u_{<\fa})_{|\fa|}}-y_{|\fa|'+c(\fa)-r(\fa)}\right),
\end{equation}
where $r(\fa)$ and $c(\fa)$ are the row and column numbers of
$\fa$ considered as entries of $\gm$ (see Figure
\ref{f.apic10_trev}),  and $|\fa|'=d+1-|\fa|$ (we use the absolute
value symbol, $|a|$, to stress that we are interested in the
integer {\it value} of the barred entry $a$). As usual, the
trivial product is defined to be $1$. The main result of this
paper is the following

\begin{thm}\label{t.lr_coeff}
$\displaystyle c\lmn=\sum_{L\in\cL\cR_{\gl,\gm}^{\gn}} c_L$.
\end{thm}

\begin{ex} Let $L$ be the equivariant Littlewood-Richardson skew tableau of Figure
\ref{f.apic11_tyam}.
Suppose that $d=4$. Consider the entry $\fa=\ol{1}$ in row 2,
column 2 of $\gm$. We have $L^u_{<\fa}=1,1,2,3,2,4$, so
$\go(L^u_{<\fa})=(2,2,1,1)$. Thus
$|\fa|'+\go(L^u_{<\fa})_{|\fa|}=(d+1-(1))+(2,2,1,1)_{1}=4+2=6$.
Also, $|\fa|'+c(\fa)-r(\fa)=(d+1-(1))+2-2=4$. Therefore the
contribution of this entry to $c_L$ is $y_6-y_4$.

Similarly, one computes the contribution of the entry $\ol{2}$ in
row 1, column 3 to be $y_5-y_5$ and the contribution of the entry
$\ol{3}$ in row 2, column 1 to be $y_3-y_1$. Therefore
$c_{L}=(y_5-y_5)(y_6-y_4)(y_3-y_1)$, which
equals $0$. 

\end{ex}

\subsection{Nonnegativity and Positivity}

If $L\in\cL\cR\lmn$, then we write $c_L>0$ if each factor in
(\ref{e.c_L}) is of the form $y_i-y_j$ with $i>j$. We write
$c_{L}\geq 0$ if either $c_{L}>0$ or $c_{L}=0$.
\begin{prop}\label{p.nonnegativity} If $L\in\cL\cR\lmn$, then $c_{L}\geq 0$.
\end{prop}

\ni Let $\cL\cR\lmnp$ be the set of $L\in\cL\cR\lmn$ for which
$c_L>0$. By Proposition \ref{p.nonnegativity}, we can restrict the
summation in Theorem \ref{t.lr_coeff} to such $L$:
\begin{cor}\label{c.lr_coeff}
$\displaystyle c\lmn=\sum_{L\in\cL\cR\lmnp} c_L$.
\end{cor}

\ni One could, of course, use (\ref{e.c_L}), the definition of
$c_{L}$, to distinguish between $c_{L}>0$ and $c_{L}=0$: $c_{L}>0$
if and only if $\go(L^u_{<\fa})_{|\fa|}>c(\fa)-r(\fa)$ for all
barred $\fa\in L$. The following Proposition gives a number of
other tests for more efficiently making this determination.
\begin{prop}\label{p.positivity_crit} If $L\in\cL\cR\lmn$, then the following are equivalent:

\ni 1. $c_{L}>0$

\ni 2. $\go(L^u_{<\fa})_{|\fa|}> c(\fa)-r(\fa)$ for all barred
$\fa\in L$.

\ni 3. $\go(L^u_{<\fa})_{|\fa|}> c(\fa)-r(\fa)$ for all barred
$\fa\in L$ with $r(\fa)=1$.

\ni 4. $\go(L^u_{<\fa})_{|\fa|}\geq c(\fa)$ for all barred $\fa\in
L$

\ni 5. $\go(L^u_{<\fa})_{|\fa|}\geq c(\fa)$ for all barred $\fa\in
L$ with $r(\fa)=1$.
\end{prop}
\ni
If $L\in\cL\cR\lmn$ satisfies any of these equivalent conditions,
then we say that $L$ is \textbf{positive}. It is obvious that
$4\implies 2\implies 3\iff 5$.  Condition 3 states that it
suffices to check barred entries on the bottom row of $L$ for
positivity. Condition 4 has the following interpretation: for any
barred entry $a\in L$, the corresponding factor $y_i-y_j$ in
$c_{L}$ satisfies $i-j\geq r(a)$ (which of course implies $i-j>0$,
the condition required for positivity).

\begin{ex}\label{e.tab_lr1} Let $d=3$, $\gl=(1,1)$, $\gm=(3,2)$, and $\gn=(3,2,1)$.
We list all $L\in\cL\cR\lmnp$, and for each $L$ we give $c_{L}$:
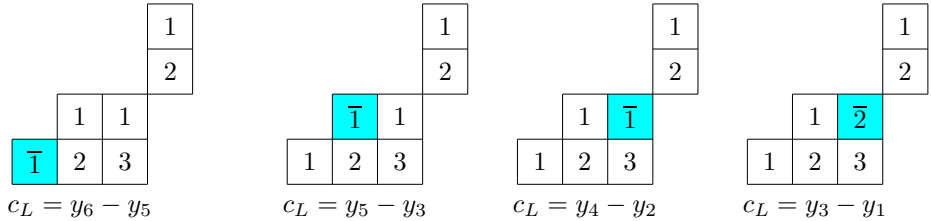
\begin{center}
\psset{unit=.6cm,linewidth=.02}

$
\begin{array}{c@{\hspace{3.5em}}c@{\hspace{3.5em}}c@{\hspace{3.5em}}c}
\pspicture(1,-1)(4,4)
\psset{linecolor=\equivpiece}
\psframe*(0,0)(1,1)
\psset{linecolor=black}
\psset{linewidth=.02}
\psline(1,2)(3,2) \psline(0,1)(3,1) \psline(0,0)(3,0)
\psline(0,0)(0,1) \psline(1,0)(1,2) \psline(2,0)(2,2)
\psline(3,0)(3,2)
\rput{0}(.5,.5){$\ol{1}$} \rput{0}(1.5,.5){$2$}
\rput{0}(2.5,.5){$3$}
\rput{0}(1.5,1.5){$1$} \rput{0}(2.5,1.5){$1$}

\psline(3,2)(4,2) \psline(3,3)(4,3) \psline(3,4)(4,4)
\psline(3,2)(3,4) \psline(4,2)(4,4)
\rput{0}(3.5,2.5){$2$} \rput{0}(3.5,3.5){$1$}

\rput{0}(1.5,-.5){$c_L=y_6-y_5$}
\endpspicture
&

\pspicture(0,-1)(3,2)
\psset{linecolor=\equivpiece}
\psframe*(1,1)(2,2)
\psset{linecolor=black}
\psline(1,2)(3,2) \psline(0,1)(3,1) \psline(0,0)(3,0)
\psline(0,0)(0,1) \psline(1,0)(1,2) \psline(2,0)(2,2)
\psline(3,0)(3,2)
\rput{0}(.5,.5){$1$} \rput{0}(1.5,.5){$2$}
\rput{0}(2.5,.5){$3$}
\rput{0}(1.5,1.5){$\ol{1}$} \rput{0}(2.5,1.5){$1$}

\psline(3,2)(4,2) \psline(3,3)(4,3) \psline(3,4)(4,4)
\psline(3,2)(3,4) \psline(4,2)(4,4)
\rput{0}(3.5,2.5){$2$} \rput{0}(3.5,3.5){$1$}

\rput{0}(1.5,-.5){$c_L=y_5-y_3$}
\endpspicture
&

\pspicture(0,-1)(3,2)
\psset{linecolor=\equivpiece}
\psframe*(2,1)(3,2)
\psset{linecolor=black}
\psset{linewidth=.02}
\psline(1,2)(3,2) \psline(0,1)(3,1) \psline(0,0)(3,0)
\psline(0,0)(0,1) \psline(1,0)(1,2) \psline(2,0)(2,2)
\psline(3,0)(3,2)
\rput{0}(.5,.5){$1$} \rput{0}(1.5,.5){$2$}
\rput{0}(2.5,.5){$3$}
\rput{0}(1.5,1.5){$1$} \rput{0}(2.5,1.5){$\ol{1}$}

\psline(3,2)(4,2) \psline(3,3)(4,3) \psline(3,4)(4,4)
\psline(3,2)(3,4) \psline(4,2)(4,4)
\rput{0}(3.5,2.5){$2$} \rput{0}(3.5,3.5){$1$}

\rput{0}(1.5,-.5){$c_L=y_4-y_2$}
\endpspicture
&

\pspicture(0,-1)(3,2)
\psset{linecolor=\equivpiece}
\psframe*(2,1)(3,2)
\psset{linecolor=black}
\psset{linewidth=.02}
\psline(1,2)(3,2) \psline(0,1)(3,1) \psline(0,0)(3,0)
\psline(0,0)(0,1) \psline(1,0)(1,2) \psline(2,0)(2,2)
\psline(3,0)(3,2)
\rput{0}(.5,.5){$1$} \rput{0}(1.5,.5){$2$}
\rput{0}(2.5,.5){$3$}
\rput{0}(1.5,1.5){$1$} \rput{0}(2.5,1.5){$\ol{2}$}

\psline(3,2)(4,2) \psline(3,3)(4,3) \psline(3,4)(4,4)
\psline(3,2)(3,4) \psline(4,2)(4,4)
\rput{0}(3.5,2.5){$2$} \rput{0}(3.5,3.5){$1$}

\rput{0}(1.5,-.5){$c_L=y_3-y_1$}
\endpspicture
\end{array}
$
\end{center}

Note that if $L$ has an unbarred $2$ in the upper right box of
$\gm$, then the unbarred column word of $L$ is not Yamanouchi, and
if $L$ has two unbarred $1$'s on the top row of $\gm$ and is not
the leftmost diagram, then $c_{L}=0$; thus we do not include such
$L$ among $\cL\cR\lmnp$. By Corollary \ref{c.lr_coeff},
$c_{\gl,\gm}^{\gn}=(y_6-y_5)+(y_5-y_3)+(y_4-y_2)+(y_3-y_1)$$=(y_6-y_1)+(y_4-y_2)$.

We list all $L\in\cL\cR_{\gm,\gl}^{\gn\, +}$, and for each $L$ we
give $c_{L}$:

\begin{center}
\psset{unit=.6cm,linewidth=.02}

$
\begin{array}{c@{\hspace{6em}}c}
\pspicture(0,-1)(4,4)
\psset{linecolor=\equivpiece}
\psframe*(0,1)(1,2)
\psset{linecolor=black}
\psset{linewidth=.02}
\psline(0,2)(1,2) \psline(0,1)(1,1) \psline(0,0)(1,0)
\psline(0,0)(0,2) \psline(1,0)(1,2)
\rput{0}(.5,.5){$3$}
\rput{0}(.5,1.5){$\ol{1}$}

\psline(1,2)(3,2) \psline(1,3)(4,3) \psline(1,4)(4,4)
\psline(1,2)(1,4) \psline(2,2)(2,4)
\psline(3,2)(3,4) \psline(4,3)(4,4)
\rput{0}(1.5,2.5){$2$} \rput{0}(2.5,2.5){$2$}
\rput{0}(1.5,3.5){$1$} \rput{0}(2.5,3.5){$1$}
\rput{0}(3.5,3.5){$1$}

\rput{0}(.5,-.5){$c_L=y_6-y_2$}
\endpspicture
&

\pspicture(0,-1)(1,2)
\psset{linecolor=\equivpiece}
\psframe*(0,1)(1,2)
\psset{linecolor=black}
\psset{linewidth=.02}
\psline(0,2)(1,2) \psline(0,1)(1,1) \psline(0,0)(1,0)
\psline(0,0)(0,2) \psline(1,0)(1,2)
\rput{0}(.5,.5){$3$}
\rput{0}(.5,1.5){$\ol{2}$}

\psline(1,2)(3,2) \psline(1,3)(4,3) \psline(1,4)(4,4)
\psline(1,2)(1,4) \psline(2,2)(2,4)
\psline(3,2)(3,4) \psline(4,3)(4,4)
\rput{0}(1.5,2.5){$2$} \rput{0}(2.5,2.5){$2$}
\rput{0}(1.5,3.5){$1$} \rput{0}(2.5,3.5){$1$}
\rput{0}(3.5,3.5){$1$}

\rput{0}(.5,-.5){$c_L=y_4-y_1$}
\endpspicture
\end{array}
$
\end{center}

By Corollary \ref{c.lr_coeff},
$c_{\gm,\gl}^{\gn}=(y_6-y_2)+(y_4-y_1)$. We see that
$c_{\gm,\gl}^{\gn}=c\lmn$. This is a general fact ensured by
(\ref{e.fact_lr_defn}); however, it is not apparent from the
statement of Corollary \ref{c.lr_coeff}.
\end{ex}

\begin{ex}
For cases where $\gm=\gn$, a formula for $c\lmn$ which produces a
different positive expression than Corollary \ref{c.lr_coeff}
appears in \cite{Bi}, \cite{Ik-Na}, and \cite{Kr}. For example,
using this formula, for $d=3$, $\gl=(2,1)$, and $\gm=\gn=(3,3,1)$,
one computes:
\begin{align*}
c\lmn=&(y_6-y_1)(y_6-y_3)(y_5-y_1)+(y_6-y_1)(y_5-y_4)(y_5-y_1).
\intertext{Using Corollary \ref{c.lr_coeff}:}
c\lmn=&(y_5-y_3)(y_5-y_1)(y_3-y_1)+(y_6-y_4)(y_5-y_1)(y_3-y_1)\\
&+(y_6-y_4)(y_6-y_3)(y_3-y_1)+(y_5-y_3)(y_4-y_3)(y_5-y_1)\\
&+(y_6-y_4)(y_4-y_3)(y_5-y_1)+(y_6-y_4)(y_6-y_3)(y_4-y_3)\\
&+(y_6-y_4)(y_5-y_4)(y_5-y_1)+(y_6-y_4)(y_5-y_4)(y_6-y_3)\\[1em]
c_{\gm,\gl}^{\gn}=&(y_6-y_4)(y_6-y_2)(y_5-y_2)+(y_5-y_3)(y_6-y_2)(y_5-y_2)\\
&+(y_6-y_4)(y_6-y_2)(y_2-y_1)+(y_5-y_3)(y_6-y_2)(y_2-y_1)\\
&+(y_6-y_4)(y_5-y_1)(y_2-y_1)+(y_5-y_3)(y_5-y_1)(y_2-y_1).
\end{align*}
\ni These three polynomials are, of course, equal.
\end{ex}


For $L\in\cL\cR\lmn$, $|\gl|+|\gm|-|\gn|=\#(\text{entries of
}L)-\#(\text{unbarred entries of }L)=\#(\text{barred entries of
}L)=\deg(c\lmn)$. In particular, if $|\gl|+|\gm|-|\gn|=0$, then
$L$ has no barred entries. When $|\gl|+|\gm|-|\gn|=0$, Theorem
\ref{t.lr_coeff} is the ordinary Littlewood-Richardson rule (see
\cite{Fu}, \cite{Li-Ri}, \cite{Sa}).

\subsection{Defining the Structure Constants $C\lmn$ for products of
two Schubert Classes in $H_T^*(Gr_{d,n})$}

The \textbf{Grassmannian} $Gr_{d,n}$ is the set of $d$-dimensional
complex subspaces of $\bC^n$.  Let $\{e_1,\ldots,e_n\}$ be the
standard basis for $\bC^n$. Consider the opposite standard flag,
whose $i$-th space is $\text{Span}(e_n,\ldots,e_{n-i+1})$.  For
$\gl\in\cP_{d,n}$, the (opposite) Schubert variety $X_\gl$ of
$Gr_{d,n}$ is defined by incident relations:
$$X_{\gl}=\{V\in Gr_{d,n}\mid \dim(V\cap F_i)\geq\dim(\bC^\gl\cap
F_i)\},\ i=1,\ldots,n,$$ where
$C^{\gl}=\text{Span}(e_{\gl_d+d},\ldots,e_{\gl_1+1})$. The
Schubert variety $X_\gl$ is invariant under the action of the
group $T=(\bC^*)^n$ on $Gr_{d,n}$. Thus it determines a class
$S_\gl$ in the equivariant cohomology ring $H_T^*(Gr_{d,n})$.

Let $V=Gr_{d,n}\times \bC^n$ be the trivial vector bundle on
$Gr_{d,n}$, with diagonal $T$-action, where $T$ acts naturally on
$Gr_{d,n}$ and on $\bC^n$ (thus $V$ is not equivariantly trivial).
Let $S=\{(w,v)\in V\mid v\in w\}$ be the tautological vector
bundle on $Gr_{d,n}$. Then $S$ is a $T$-invariant sub-bundle of
$V$. Let $Y_1,\ldots,Y_n$ be the equivariant Chern roots of $V^*$
and $X_1,\ldots,X_d$ the equivariant Chern roots of $S$. Then
$H_T^*(Gr_{d,n})$ is a free $\bC[Y_1,\ldots,Y_n]$-module, with the
Schubert classes forming a $\bC[Y_1,\ldots,Y_n]$-basis. Thus for
$\gl,\gm\in\cP_{d,n}$,
$$S_\gl S_\gm=\sum C_{\gl,\gm}^\gn S_\gn,$$ for some $C_{\gl,\gm}^{\gn}\in
\bC[Y_1,\ldots,Y_n]$, where the summation is over all
$\gn\in\cP_{d,n}$. We have (see \cite{Fu1}, \cite{Kn-Ta},
\cite{Mi})

\begin{prop}
For $\gl\in\cP_{d}$,
$S_\gl=s_\gl(X_1,\ldots,X_d,-Y_n,\ldots,-Y_1,0,0,\ldots)$.
\end{prop}

\ni Thus by specializing (\ref{e.fact_lr_defn}), we can determine
the structure constants $C\lmn$.

\begin{cor}\label{c.c_C} For $\gl,\gm,\gn\in \cP_{d,n}$,
$C_{\gl,\gm}^{\gn}=c_{\gl,\gm}^{\gn}(-Y_n,\ldots,-Y_1,0,0,\ldots)$.
\end{cor}

\subsection{Computing the Structure Constants $C\lmn$}
Let $\gl,\gm,\gn\in\cP_{d,n}$.  By Corollary \ref{c.c_C}, the
structure constant $C\lmn$ can be computed using the formula for
$c\lmn$. Let $L\in\cL\cR\lmn$. Define
\begin{equation}\label{e.C_L}
\begin{split}
C_{L}&=c_{L}(-Y_n,\ldots,-Y_1)\\
&=\prod_{\fa\in L\atop \fa\text{
barred}}\left(Y_{(n-d)+|\fa|-(c(\fa)-r(\fa))}-Y_{(n-d)+|\fa|-\go(L^u_{<\fa})_{|\fa|}}\right).
\end{split}
\end{equation}

We write $C_{L}>0$ if each factor in (\ref{e.C_L}) is of the form
$Y_i-Y_j$ with $i>j$, and we write $c_{L}\geq 0$ if either
$C_{L}>0$ or $C_{L}=0$. By (\ref{e.C_L}), $c_{L}\geq 0\iff
C_{L}\geq 0$, and $c_{L}=0\iff C_{L}=0$. Thus Propositions
\ref{p.nonnegativity} and \ref{p.positivity_crit} imply
\begin{cor}\label{c.nonnegativity}
$C_{L}\geq 0$, and $C_{L}>0\iff$ $L$ satisfies any of the
equivalent conditions of Proposition \ref{p.positivity_crit}.
\end{cor}
\ni By Theorem \ref{t.lr_coeff}, Corollary \ref{c.c_C}, and
Corollary \ref{c.nonnegativity}, we have
\begin{cor}\label{c.equiv_lr_coeff} $\displaystyle
C\lmn=\sum_{L\in\cL\cR_{\gl,\gm}^{\gn}}
C_{L}=\sum_{L\in\cL\cR\lmnp} C_{L}$.
\end{cor}

\begin{ex} We continue Example \ref{e.tab_lr1}.
For $n\geq 6$, $\gl,\gm\in\cP_{d,n}$. Thus for $\gn\in\cP_{d,n}$,
$C\lmn=(Y_{n+1-2}-Y_{n+1-6})+(Y_{n+1-1}-Y_{n+1-4})=
(Y_{n-1}-Y_{n-5})+(Y_{n}-Y_{n-3})$.
\end{ex}

\subsection{Equivalence of Molev's Results}
Our equivariant Littlewood-Richardson skew tableaux are
essentially the same as the Molev's indexing tableaux \cite{Mo1}.
To determine the tableau in \cite{Mo1} which corresponds to our
$L\in\cL\cR\lmn$, replace all barred entries of $L|_{\gm}$ by
unbarred entries and visa-versa, and then rotate the resulting
object by 180 degrees. If one makes this modification, then
Corollary \ref{c.lr_coeff} is equivalent to \cite[Theorem
2.1]{Mo1} after accounting for the relationship between double
Schur polynomials and factorial Schur functions (see
\cite[(1.9)]{Mo1}), and Corollary \ref{c.equiv_lr_coeff} is
identical to \cite[Corollary 3.1]{Mo1}.

In our notation, Molev's positivity criterion states that for
$L\in\cL\cR\lmn$, $c_{L}>0$ if and only if
\begin{equation}\label{e.molev_positivity}
\go(L^u)'_{c(\fa)}\geq |\fa| \text{ for all }\fa\in L\text{ with
}r(\fa)=1,
\end{equation}
where $\go(L^u)'$ is the {\it conjugate partition} to $\go(L^u)$
(in this case Molev calls $L|_\gm$ {\it $\gn$-bounded}). One can
re-express (\ref{e.molev_positivity}) as follows:
\begin{equation*}
\go(L^u)_{|\fa|}\geq c(\fa) \text{ for all }\fa\in L\text{ with
}r(\fa)=1.
\end{equation*}
It is not difficult to see that this condition is equivalent to
Proposition \ref{p.positivity_crit}.5.

Related and more general results have been achieved in several
directions. Robinson \cite{Ro} has given a Pieri rule in the
equivariant cohomology of the flag variety. McNamara \cite{Mc}
introduced factorial Grothendieck polynomials, generalizations of
factorial Schur functions, and has given a rule for computing the
structure constants for various of their products.

This paper is organized as follows. In Section \ref{s.tableaux},
we introduce various types of tableaux which will appear
throughout the paper. In Section \ref{s.proof_nonnegativity}, we
prove Propositions \ref{p.nonnegativity} and
\ref{p.positivity_crit}, the nonnegativity property and positivity
criteria of $c_L$. In Section \ref{s.proof_stembridge}, we outline
the main steps in our proof of Theorem \ref{t.lr_coeff}, whose two
difficult technical lemmas are proved in Sections
\ref{s.proof_lemma1} and \ref{s.proof_lemma2}. In Section
\ref{s.hatted_tableaux}, we define a set of involutions required
for the proofs of these two lemmas. In Section
\ref{s.bij_knutson-tao}, we describe a bijection between positive
equivariant Littlewood-Richardson skew tableaux and Knutson-Tao
puzzles.

\section{Several Types of Tableaux}\label{s.tableaux}

In this section we collect the definitions of the several types of
tableaux which we will encounter in the remainder of the paper:
reverse barred tableaux, reverse barred subtableaux, and reverse
hatted tableaux.  The latter two are refinements of the first.

A \textbf{reverse barred tableau of shape $\gm$} is a skew barred
tableau of shape $\es*\gm$; alternatively, it can be defined as a
reverse Young diagram of shape $\gm$, each of whose boxes is
filled with either an integer $k$ or a barred integer $\ol{k}$,
$k\in\{1,\ldots,d\}$, in such a way that the values of the
entries, without regard to whether or not they are barred, weakly
increase along any row from left to right and strictly increase
along any column from top to bottom. We denote the set of all
reverse barred tableaux of shape $\gm$ by $\cB(\gm)$. If
$B\in\cB(\gm)$, then define $\gl*B$ to be the skew barred tableau
obtained by placing the Young tableau whose $i$-th row consists of
$\gl_i$ $i$'s above and to the right of $B$. Then $B\mapsto \gl*B$
defines a bijection from $\{B\in\cB(\gm)\mid (\gl*B)^u$ is
Yamanouchi$\}$ to the equivariant Littlewood-Richardson skew
tableaux of shape $\gl*\gm$, whose inverse map is $L\mapsto
L|_\gm$. Any $a\in B$ also corresponds to an entry $a\in\gl*B$.
Define $B^u$ and $B^u_{<a}$ to be $(\es*B)^u$ and $(\es*B)^u_{<a}$
respectively.

A \textbf{reverse barred subtableaux of shape $\gm$} is a reverse
Young diagram $\gm$ each of whose boxes contains either an integer
$k$, a barred integer $\ol{k}$, or is empty, where
$k\in\{1,\ldots,d\}$. A \textbf{reverse subtableau of shape $\gm$}
is a reverse barred tableau of shape $\gm$ which has no barred
entries. We do not define any notion of row semistrictness or
column strictness for such objects, as no such conditions will be
required for our purposes. Denote the set of all reverse
subtableaux and reverse barred subtableaux of shape $\gm$ by
$\cR_{sub}(\gm)$ and $\cB_{sub}(\gm)$ respectively. We have the
following containments:
\begin{equation*}
\begin{array}{ccc}
\cR_{sub}(\gm)&\subset&\cB_{sub}(\gm)\\[.5em]
\cup&&\cup\\[.5em]
\cR(\gm)&\subset&\cB(\gm)
\end{array}
\end{equation*}
For $B\in\cB_{sub}(\gm)$ and $a\in B$, define $B^u$ and $B^u_{<a}$
just as for elements of $\cB(\gm)$, assuming that when reading the
unbarred column word of $B$, both barred entries and empty boxes
are skipped over. If $B\in \cB_{sub}(\gm)$, then define $\wt{B}\in
\cR_{sub}(\gm)$ to be the reverse subtableau obtained by removing
all bars from entries of $B$, i.e., replacing each barred entry of
$B$ by an unbarred entry of the same value.

A \textbf{reverse hatted tableau of shape $\gm$} is a reverse
Young diagram $\gm$ each of whose boxes is filled with either a(n)
(un-hatted) integer $k$, a {\it left} hatted integer $\check{k}$,
or a {\it right} hatted integer $\hat{k}$, $k\in\{1,\ldots,d\}$,
such that the values of the entries, without regard to whether or
not they are hatted, weakly increase along any row from left to
right and strictly increase along any column from top to bottom.
Denote the set of all reverse hatted tableaux of shape $\gm$ by
$\cH(\gm)$.  If $H$ is a reverse hatted tableau, then define
$\ol{H}$ to be the reverse barred tableau produced by replacing
all hats (right and left) by bars. Hence for a reverse barred
tableau $B$ with $m$ barred entries, there are $2^m$ reverse
hatted tableaux $H$ such that $\ol{H}=B$ (since each $\ol{k}$ of
$B$ can be replaced by either $\check{k}$ or $\hat{k}$). For $a\in
H$, define $H^u$ and $H^u_{<a}$ to be $\ol{H}^u$ and
$\ol{H}^u_{<a}$ respectively. Define $H^l$ (resp. $H^r$) to be the
set of left-hatted (resp. right-hatted) entries of $H$.

We next give two different ways to generalize the polynomial $c_L$
defined in Section \ref{s.results}. Let $\gx\in\bN^d$. For
$B\in\cB_{sub}(\gm)$, define
\begin{equation}\label{e.wc_B}
c_{\gx,B}=\prod_{\fa\in B\atop a\text{
barred}}(y_{e_{\gx,B}(a)}-y_{f_B(a)}),
\end{equation}
where $e_{\gx,B}(\fa):=(\gx+\go(B^u_{<\fa}))_{|\fa|}$ and
$f_B(\fa):=|a|'+c(a)-r(a)$, $\fa\in B$. For $H\in\cH(\gm)$, define
\begin{equation}\label{e.hatted_coeff}
d_{\gx,H}=\prod_{\fa\in H^{l}}y_{e_{\gx,H}(\fa)}\prod_{\fa\in
H^{r}}(-y_{f_H(\fa)}),
\end{equation}
where $e_{\gx,H}(\fa):=(\gx+\go(H^u_{<\fa}))_{|\fa|}$ and
$f_H(\fa):=|\fa|'+c(\fa)-r(\fa)$, $\fa\in H$. In both
(\ref{e.wc_B}) and (\ref{e.hatted_coeff}), the empty product is
defined to equal 1.

Let $B\in\cB(\gm)$.  By definition,
\begin{equation}\label{e.star_nostar}
c_{\gl*B}=c_{\gl+\gr+1,B}.
\end{equation}
In addition, the equation
\begin{equation}\label{e.lr_coeff_hatted_coeff}
c_{\gx,B}=\sum_{H\in\cH(\gm)\atop \ol{H}=B}d_{\gx,H}
\end{equation}
expresses $c_{\gx,B}$ by expanding (\ref{e.wc_B}) in terms of
monomials in the $y_i$'s. Combining (\ref{e.star_nostar}) and
(\ref{e.lr_coeff_hatted_coeff}), we have
\begin{equation}\label{e.lr_coeff_h_b}
c_{\gl*B}=\sum_{H\in\cH(\gm)\atop \ol{H}=B}d_{\gl+\gr+1,H}.
\end{equation}

If $R\in\cR_{sub}(\gm)$, then define $\xy^R=\prod_{\fa\in
R}\left(x_{\fa}-y_{f_R(\fa)}\right)$. This definition is
consistent with the definition of $\xy^R$, $R\in\cR(\gm)$, given
in Section \ref{s.results}.

\section{Proofs of Nonnegativity Property and Positivity Criteria}\label{s.proof_nonnegativity}

Let $L\in\cL\cR\lmn$, and let $B=L|_\gm$. For $\fa\in B$, which we
also view as an entry of $L$, define $L^u_{\leq \fa}$ to be
$L^u_{< \fa}$ if $\fa$ is barred, or $L^u_{< \fa}$ appended with
$\fa$ if $\fa$ is not barred. Define
\begin{equation*}
\gD(\fa)=\go(L^u_{\leq\fa})_{|\fa|}-c(\fa)+r(\fa).
\end{equation*}
If $\fa$ is barred, then $\go(L^u_{\leq\fa})=\go(L^u_{<\fa})$;
hence $\gD(\fa)$ gives the difference between the two indices
$i-j$ of the factor $y_i-y_j$ corresponding to $\fa$ in
(\ref{e.c_L}). Therefore Propositions \ref{p.nonnegativity} and
\ref{p.positivity_crit} are equivalent to the following two lemmas
respectively.
\begin{lem}\label{l.tech_nonneg}
If $\gD(\fa)<0$ for some barred $\fa\in B$, then $\gD(\fb)=0$ for
some barred $\fb\in B$.
\end{lem}

\begin{lem}\label{l.tech_pos_crit} The following are equivalent:\\
(i) $\gD(a)> 0$ for all barred $\fa\in B$.

\ni (ii) $\gD(\fa)>0$ for all barred $\fa\in B$ with $r(\fa)=1$.

\ni (iii) $\gD(a)\geq r(\fa)$ for all barred $\fa\in B$.
\end{lem}

\ni Before proving these two lemmas, we first establish some
properties of $\gD$.

\begin{lem}\label{l.tech.np} The function $\gD:B\to\bZ$ satisfies the following
properties:

\ni (i) If $a\in B$ and $c(a)=1$, then $\gD(a)\geq 0$, with
equality implying that $a$ is barred.

\ni (ii) If one moves left by one box, then $\gD$ can decrease by
at most one. If it does decrease by one, then the left box must be
barred.

\ni (iii) If $\gD(\fa)\leq 0$ for some $\fa\in B$, then
$\gD(\fb)=0$ for some barred $\fb\in B$ on the same row as $\fa$.

\ni (iv) The function $a\mapsto \gD(a)-r(a)$ is weakly decreasing
as one moves down along any column.
\end{lem}
\begin{proof}
(i) Since $r(\fa)\geq 1$, $\gD(\fa)\geq 0$. If $\gD(a)=0$, then
$r(\fa)=1$ and $\go(L^u_{\leq a})_{|a|}=0$. The latter requirement
implies that $a$ is barred.

\ni (ii) If entry $m$ lies one box left of $\fa$, then
$-c(m)=-c(\fa)-1$, $r(m)=r(\fa)$, and $\go(L^u_{\leq m})_{|m|}\geq
\go(L^u_{\leq\fa})_{|m|}\geq \go(L^u_{\leq\fa})_{|\fa|}$, where
the first inequality is an equality if and only if $m$ is barred.
The second inequality is a consequence of the fact that the
unbarred column word of $L$ is Yamanouchi.

\ni (iii) Let $m$ be rightmost entry in the same row as $\fa$. If
$\gD(m)=0$, then by (i), $m$ is barred, so letting $\fb=m$ we are
done. Otherwise $\gD(m)>0$. By (ii), as one moves left from $m$ to
$a$ along the row the two entries lie on, one must encounter some
barred $\fb$ for which $\gD(b)=0$.

\ni (iv)  If entry $m$ lies one box below $\fa$, then
$\go(L^u_{\leq\fa})_{|\fa|}=\go(L^u_{\leq m})_{|\fa|}\geq
\go(L^u_{\leq m})_{|m|}$, since the unbarred column word of $L$ is
Yamanouchi.
\end{proof}

\begin{proof}[Proof of  Lemmas \ref{l.tech_nonneg} and \ref{l.tech_pos_crit}]
Lemma \ref{l.tech_nonneg} is a special case of Lemma
\ref{l.tech.np}(iii). In Lemma \ref{l.tech_pos_crit}, implications
(iii) $\implies$ (i) $\implies$ (ii) are clear. We prove (ii)
$\implies$ (iii). Suppose that $a\in B$ is a barred entry such
that $\gD(\fa)<r(\fa)$. Let $m$ be the bottom entry in column
$c(\fa)$. By Lemma \ref{l.tech.np}(iv), $\gD(m)<r(m)$. Since
$r(m)=1$, $\gD(m)\leq 0$. By Lemma \ref{l.tech.np}(iii),
$\gD(b)=0$ for some barred $b$ on the bottom row of $B$.
\end{proof}

\section{Generalization of Stembridge's Proof}\label{s.proof_stembridge}

In this section we list the main steps in the proof of Theorem
\ref{t.lr_coeff}.  The bulk of the technical work, however, namely
the proofs of Lemmas \ref{l.lra} and \ref{l.bad_guys_vanish}, is
taken up in the three subsequent sections. The underlying logic
and structure of our arguments in this and the following three
sections follows Stembridge \cite{Stem1}, who works out similar
results for ordinary Schur functions.

For $k\in \bN$, define the polynomial
$(x_j\,|\,y)^k=(x_j-y_1)\cdots(x_j-y_k)$. For
$\gx=(\gx_1,\ldots,\gx_d)\in\bN^d$, define
$a_\gx(x\,|\,y)=\det[(x_j\,|\,y)^{\gx_i}]_{1\leq i,j\leq d}$.

\begin{lem}\label{l.lra}  $\displaystyle a_{\gl+\gr}\xy s_{\gm}\xy=\sum\limits_{B\in
\cB(\gm)}c_{\gl*B} a_{\gl+\gr+\go(B^u)}\xy$.
\end{lem}

\begin{lem}\label{l.bad_guys_vanish} $\displaystyle \sum c_{\gl*B}
a_{\gl+\gr+\go(B^u)}\xy=0$, where the sum is over all
$B\in\cB(\gm)$ such that the unbarred column word of $\gl*B$ is
not Yamanouchi.
\end{lem}

\ni The following four corollaries follow easily from these two
lemmas.

\begin{cor}\label{c.lr2}
$\displaystyle a_{\gl+\gr}\xy s_{\gm}\xy=\sum c_{\gl*B}
a_{\gl+\gr+\go(B^u)}\xy$, where the sum is over all $B\in\cB(\gm)$
such that the unbarred column word of $\gl*B$ is Yamanouchi.
\end{cor}
\ni Suppose that $B\in\cB(\gm)$ is such that the unbarred column
word of $\es*B$ is Yamanouchi. If $B$ has barred entries, then by
Propositions \ref{p.nonnegativity} and \ref{p.positivity_crit}.5,
$c_{\es*B}= 0$. If $B$ has no barred entries, then  $B$ must be
the unique reverse tableau of shape $\gm$ and  content $\gm$: $B$
contains a 1 at the top of each column, and its entries increase
by 1 per box as one moves down any column. Thus, by setting
$\gl=\es$ in Corollary \ref{c.lr2}, we arrive at a new proof of
the bialternant formula for the factorial Schur function
(\cite{Go-Gr}, \cite{Mac1}):
\begin{cor}\label{c.alternant}
$\displaystyle s_{\gm}\xy=a_{\gm+\gr}\xy/a_\gr\xy$.
\end{cor}
\ni Dividing both sides of the equation in Corollary \ref{c.lr2}
by $a_\gr\xy$ and applying Corollary \ref{c.alternant} yields
\begin{cor}\label{c.lr_rule}
$s_\gl\xy s_{\gm}\xy=\sum c_{\gl*B} s_{\gl+\go(B^u)}\xy$, where
the sum is over all $B\in\cB(\gm)$ such that the unbarred column
word of $\gl*B$ is Yamanouchi.
\end{cor}
\ni Regrouping the terms in this summation:
\begin{equation*}
s_\gl\xy s_\gm\xy =
\sum_\gn\left(\sum_{B\in\cB(\gm)\atop{(\gl*B)^u\text{
Yamanouchi}\atop \gl+\go(B^u)=\gn}}c_{\gl*B}\right) s_\gn\xy
=\sum_\gn\left(\sum_{L\in\cL\cR\lmn}c_L\right) s_\gn\xy.
\end{equation*}
This proves Theorem \ref{t.lr_coeff}.

\begin{rem} Let $\gk\in\cP_d$, $\gk\leq\gm$, i.e., $\gk_i\leq \gm_i$,
$i=1,\ldots,d$. One can extend our analysis to factorial skew
Schur functions of the form $s_\gmk\xy$ (see \cite{Mac1}). One
replaces $\cB(\gm)$ with $\cB(\gmk)$, the set of all reverse
barred tableaux of shape $\gmk$. All above definitions extend
naturally.  For example, for $B\in\cB(\gmk)$, $c_{\gl*B}$ is
computed just as for $B\in\cB(\gm)$, but with all boxes of
$\gk\subset\gm$ assumed to be empty. All proofs are virtually
unchanged, modified only by formally replacing $\gm$ by $\gmk$. As
a generalization of Corollary \ref{c.lr_rule}, we obtain
$$s_\gl\xy s_{\gmk}\xy=\sum c_{\gl*B} s_{\gl+\go(B^u)}\xy,$$ where the
sum is over all $B\in\cB(\gmk)$ such that $(\gl*B)^u$ is
Yamnaouchi. This generalizes Zelevinsky's extension of the
Littlewood-Richardson rule (\cite{Stem1}, \cite{Ze}).
\end{rem}

\section{Involutions on Reverse Hatted Tableaux}\label{s.hatted_tableaux}

In his proof, Stembridge \cite{Stem1} utilizes involutions on
Young tableaux introduced by Bender and Knuth \cite{Be-Kn}. There
is an analogous set of involutions on $\cH(\gm)$ which satisfy
properties required for the proofs of Lemmas \ref{l.lra} and
\ref{l.bad_guys_vanish} (see Lemma \ref{l.invol_properties}). We
remark that we were unable to find a suitable set of involutions
on $\cB(\gm)$, and this is what initially led us to examine
$\cH(\gm)$. If the involutions on $\cH(\gm)$ are restricted to
$\cR(\gm)$, then the Bender-Knuth involutions are recovered.

\subsection{The Involutions $\mathbf{s_1,\ldots,s_{d-1}}$ of $\cH(\gm)$}

Let $H\in\cH(\gm)$, and let $i\in\{1,\ldots,d-1\}$ be fixed. Then
an entry $a$ of $H$ with value $i$ or $i+1$ is
\begin{itemize}
\item \textbf{free} if there is no entry of value $i+1$ or $i$
respectively in the same column;
\item \textbf{semi-free} if there is an entry of value $i+1$ or
$i$ respectively in the same column, and at least one of the two
is hatted; or
\item \textbf{locked} if there is an entry of value $i+1$ or $i$
respectively in the same column, and both entries are unhatted.
\end{itemize}
Note that any entry of value $i$ or $i+1$ must be exactly one of
these three types, and each hatted entry of value $i$ or $i+1$
must be either free or semi-free. In any row, the free entries are
consecutive. Semi-free entries come in pairs, one below the other,
as do locked entries.

To define the action of $s_i$ on $H\in\cH(\gm)$, we first consider
how it modifies the free entries of $H$ (see Example
\ref{ex.involution}):

\begin{itemize}

\item[1.]  Let $S$ be a maximal string of free entries with values
$i$ and $i+1$ on some row of $H$. Let $S^0$, $S^l$, and $S^r$
denote the unhatted, left-hatted, and right-hatted entries of $S$
respectively. Modify $S^{\circ}\cup S^{l}$, as follows:
\begin{enumerate}
\item[A] Change the value of each entry of value $i$ to $i+1$ and
each entry of value $i+1$ to $i$, without changing whether or not
it has a left hat.

\item[B] Swap the entries of value $i$ with those of value $i+1$:
remove all entries of value $i$; then move each entry of value
$i+1$, beginning with the rightmost one, into the rightmost
available empty box; then put the removed entries of value $i$
back into the empty boxes of $B$, preserving the relative order of
barred and unbarred entries.
\end{enumerate}
In this step, $S^{\circ}\cup S^{l}$ has been modified. No other
entries of $H$, in particular no entries of $S^r$, have been
modified, changed, or moved. Denote the modified string $S$ by
$S_1$.  A potential problem has been introduced: the values of the
entries of $S_1$ may not be weakly increasing as one moves from
left to right. In step 2 we correct for this.

\item[2.] Let $(S_1^r)_i$ and $(S_1^r)_{i+1}$ denote the entries
of $S_1^r$ of value $i$ and $i+1$ respectively.  Beginning with
the leftmost entry $\fa\in (S_1^r)_i$, let $\fb$ be the entry of
$S_1$ to the left of $\fa$. If $\fb$ has value $i+1$, then switch
the entries $\fb$ and $\fa$, and then change the left entry from
$\wh{i}$ to $\wh{i+1}$. Now move right to the next entry of
$(S_1^r)_{i}$, and repeat this procedure until it has been
performed on all entries of $(S_1^r)_{i}$. Next, beginning with
the rightmost entry $\fa\in (S_1^r)_{i+1}$, let $\fb$ be the entry
of $S_1$ to the right of $\fa$. If $\fb$ has value $i$, then
switch the entries $\fb$ and $\fa$, and then change the right
entry from $\wh{i+1}$ to $\wh{i}$. Now move left to the next entry
of $(S_1^r)_{i+1}$, and repeat this procedure until it has been
performed on all entries of $(S_1^r)_{i+1}$.

Upon completion, we denote by $S_2$ the resulting string obtained
by modifying $S_1$.  It is weakly increasing.
\end{itemize}

\noindent  We next consider how $s_i$ modifies the semi-free
entries of $H$:

\begin{itemize}
\item[3.] For a semi-free pair consisting of two entries lying in
the same column of $H$, each entry removes its hat (if it has one)
and places it on top of the other entry.
\end{itemize}
The reverse tableau $s_iH$ is obtained by applying steps 1 and 2
to each maximal string $S$ of free entries of $H$ (replacing $S$
by $S_2$) and then applying step 3 to each semi-free pair.

\begin{ex}\label{ex.involution}
We illustrate steps 1 and 2.  Suppose that $i=2$, and $S$ consists
of the following maximal string of consecutive free entries lying
along some row of $H$:
\begin{align*}
S=&2\ \hat{2}\ \check{2}\ 2\ 2\ \check{2}\ \hat{2}\ 2\ 3\
\check{3}\ \hat{3}\ \check{3}\\
S^0\cup S^l=&2\ \ph{\hat{2}}\ \check{2}\ 2\ 2\ \check{2}\
\ph{\hat{2}}\ 2\ 3\ \check{3}\ \ph{\hat{3}}\ \check{3}\\
&3\ \ph{\hat{2}}\ \check{3}\ 3\ 3\ \check{3}\
\ph{\hat{2}}\ 3\ 2\ \check{2}\ \ph{\hat{3}}\ \check{2}\\
&2\ \ph{\hat{2}}\ \check{2}\ \check{2}\ 3\ \check{3}\
\ph{\hat{2}}\ 3\ 3\ \check{3}\ \ph{\hat{3}}\ 3\\
S_1=&2\ \hat{2}\ \check{2}\ \check{2}\ 3\ \check{3}\ \hat{2}\ 3\
3\
\check{3}\ \hat{3}\ 3\\
S_2=&2\ \hat{2}\ \check{2}\ \check{2}\ 3\ \hat{3}\ \check{3}\ 3\
3\
\check{3}\ \hat{3}\ 3\\
\end{align*}
In line 2 we remove the entries of $S^r$ from the picture for
convenience, in order to focus attention on the operations
performed in step 1, which only affect $S^0\cup S^{l}$. In lines 3
and 4 the results of applying steps 1A and 1B successively to
$S^0\cup S^{l}$ are shown.  In line 5, the removed entries from
$S^r$ are replaced. In line 6, the result of applying step 2 to
$S_1$ is shown. Only two entries are changed in this step.
\end{ex}

This algorithm defines maps $b_l:H^l\to (s_iH)^l$ and $b_r:H^r\to
(s_iH)^r$, as follows. If $\fa\in H^l$ is free, then in step 1A,
the value of $a$ is either increased or decreased by 1; in step
1B, it is then moved to a different box; in step 2, this new entry
in this new box is moved at most one box and changed by at most
one in value, resulting in the entry we denote by $b_l(\fa)$. If
$\fa\in H^r$ is free, then $\fa$ is unchanged in step 1 and moved
at most one box and changed by at most one in value in step 2.
Denote the resulting entry by $b_r(\fa)$.  If $\fa\in H^l$ or
$\fa\in H^r$ is semi-free, then $b_l(\fa)$ or $b_r(\fa)$ is the
entry in $s_iH$ which it gives its hat to.

In Example \ref{ex.involution}, if $a$ is the rightmost entry of
$S$, which is a $\check{3}$, then $b_l(a)$ is the $\check{2}$
which is the fourth entry of $S_2$ from left. These two entries
are, of course, entries of $H$ and $s_iH$ respectively.

\begin{lem}
$s_i$ is an involution on $\cH(\gm)$, $i\in\{1,\ldots,d-1\}$.
\end{lem}
\begin{proof} We begin by showing that $s_iH\in\cH(\gm)$, i.e.,
$s_iH$ is row semistrict and column strict. The only nonobvious
condition is that if $S$ is any maximal string of free entries of
$H$ lying along some row, and $S_2$ the string that replaces it in
$s_iH$, then $s_iH$ weakly increases along the left and right
boundaries of $S_2$. To see this, note that if any entry of $H$ of
value $i+1$ is free, then so are all entries of value $i+1$ to the
right of it in the same row; and if any entry of $H$ of value $i$
is free, then so are all entries of value $i$ to the left of it in
the same row. Thus by the maximality of $S$, there are no entries
of $H$ of value $i$ in the same row and to the right of $S$, and
there are no entries of $H$ of value $i+1$ in the same row and to
the left of $S$. Hence changing values of $S$ from $i$ to $i+1$
and visa-versa to form $S_2$ does not affect the row
semistrictness of $H$ along its boundaries.

We next show that $s_i^2=\text{id}$. Since the free entries of $H$
lie in the same boxes as the free entries of $s_iH$, it suffices
to show that $s_i^2(S)=S$ for any maximal string $S$ of free
entries of $H$ (where $s_iS$ is defined to be $s_iH$ restricted to
$S$). If step 1 is applied to $(s_iS)^{\circ}\cup (s_iS)^l$, then
one sees that the same entries of $S^{\circ}\cup S^l$ are
retrieved, although possibly not in their same boxes. However the
relative order of the entries is the same. Now one checks that for
$\fa\in H^r$, $b_r^2(\fa)=\fa$.
\end{proof}

Let $\gs_i$ be the simple transposition of the permutation group
$S_d$ which exchanges $i$ and $i+1$.  The involution $s_i$
satisfies the following properties:
\begin{lem}\label{l.invol_properties} Let $H\in\cH(\gm)$, $\fa\in
H^l$, and $\fb\in H^r$.  Then\\
(i) $|b_l(\fa)|=\gs_i|\fa|$\\
(ii) $\go((s_iH)^u)=\gs_i\go(H^u)$.\\
(iii) $\go((s_iH)^u_{<b_l(\fa)})=\gs_i\go(H^u_{<\fa})$\\
(iv) $e_{\gs_i\gx,s_iH}(b_l(\fa))=e_{\gx,H}(\fa)$\\
(v) $f_{s_iH}(b_r(b))=f_H(b)$\\
(vi) $d_{\gs_i\gx,s_iH}=d_{\gx,H}$
\end{lem}
\begin{proof}
(i), (ii), and (iii) follow from the construction of $s_i$.

\ni (iv) By parts (iii) and (i),
\begin{align*}
e_{\gs_i\gx,s_iH}(b_l(\fa))&=(\gs_i\gx+\go((s_iH)^u_{<b_l(\fa)}))_{|b_l(\fa)|}\\
&=(\gs_i\gx+\gs_i\go(H^u_{<\fa}))_{|b_l(\fa)|}\\
&=(\gs_i\gx+\gs_i\go(H^u_{<\fa}))_{\gs_i|\fa|}\\
&=(\gs_i(\gx+\go(H^u_{<\fa})))_{\gs_i|\fa|}
=(\gx+\go(H^u_{<\fa}))_{|\fa|}
=e_{\gx,H}(\fa)
\end{align*}

\ni (v) Under $b_r$, the entry $\fb$ is either kept in place,
moved up, down, left, or right by one box. In these cases, its
value is either left unchanged, decreased, increased, increased,
or decreased by one respectively.  The result now follows from the
definition of $f_H$.

\ni (vi) This is a consequence of (i), (ii), and
(\ref{e.hatted_coeff}).
\end{proof}

Let $H\in\cH(\gm)$ and let $\gs\in S_d$. Choose some decomposition
of $\gs$ into simple transpositions:
$\gs=\gs_{i_1}\cdots\gs_{i_t}$. Define $\gs H:=s_{i_1}\cdots
s_{i_t}H$. Although $\gs H$ depends on the decomposition chosen
for $\gs$, by Lemma \ref{l.invol_properties}(ii) and (vi),
\begin{equation}\label{e.perm_properties}\go((\gs H)^u)=\gs\go(H^u)\qquad\text{and}\qquad
d_{\gs\gx,\gs H}=d_{\gx,H}.
\end{equation}
In particular, both $\go((\gs H)^u)$ and $d_{\gs\gx,\gs H}$ are
independent of the decomposition of $\gs$.

\section{Proof of Lemma \ref{l.lra}}\label{s.proof_lemma1}

Lemma \ref{l.lra} is a generalization of \cite[(1)]{Stem1}. In
proving \cite[(1)]{Stem1}, Stembridge uses the simple fact that if
$S$ is a tableau and $\gx=(\gx_1,\ldots,\gx_d)\in\bN^d$, then
$x^\gx x^S=x^{\gx+\go(S)}$. The generalization of this fact which
we will need in order to prove Lemma \ref{l.lra} is the following
lemma. Define $\xy^\gx=(x_1\,|\,y)^{\gx_1}\cdots
(x_d\,|\,y)^{\gx_d}$.
\begin{lem}\label{l.pra_interm} Let $R\in\cR_{sub}(\gm)$ and let
$\gx\in\bN^d$.  Then
$$(x \mid y)^{\gx}(x \mid y)^{R}=\sum_{B\in\cB_{sub}(\gm)\atop \widetilde{B}=R} c_{\gx+1,B}\cdot(x
\mid y)^{\gx+\go(B^u)}.$$
\end{lem}

\ni In fact, we only need this lemma for $R\in\cR(\gm)$. We prove
this result more generally for $R\in\cR_{sub}(\gm)$ only to allow
for induction on the number of entries of $R$ (and thus allow for
the possibility that some boxes of $R$ are empty). We remark that
$\cR_{\text{sub}}(\gm)$ and $\cB_{\text{sub}}(\gm)$ were
introduced in this paper solely to allow for induction in this
proof.
\begin{proof} By induction on the number of entries in $R$.
Let $\fa$ be an entry of $R$ with value $k$, such that $R$ has no
entry of value $k$ in any column to the left of $\fa$. Let $\ga$
be the box containing $\fa$. Let $R'=R\setminus \fa$ be the the
reverse subtableau which results from removing $\fa$ from $R$.

If $B\in\cB_{sub}(\gm)$ is such that $\widetilde{B}=R$, then the
entry of $B$ in box $\ga$, which we denote by $B_{\ga}$, must
either be $k$ or $\ol{k}$. Let $B'$ denote $B\setminus B_\ga$. The
following three sets are in bijection with one another:
\begin{align*}
\{B\in\cB_{sub}(\gm)\mid \widetilde{B}=R,
B_{\ga}=k\}&\longleftrightarrow
\{B\in\cB_{sub}(\gm)\mid\widetilde{B}=R,
B_{\ga}=\ol{k}\}\\
&\longleftrightarrow\{D\in\cB_{sub}(\gm)\mid \widetilde{D}=R'\}.
\end{align*}
The first bijection simply adds a bar to $B_\ga$, and the second
bijection removes $B_\ga$ from $B$, mapping $B$ to $B'$. For
brevity, we denote $e_{B,\gx}(B_{\ga})$ and $f_{B}(B_{\ga})$ by
just $e(B_{\ga})$ and $f(B_{\ga})$ respectively for the remainder
of this proof. If $B_{\ga}$ is unbarred, then
\begin{equation*}
c_{\gx+1,B}=c_{\gx+1,B'}\quad \text{and}\quad
(x\mid y)^{\gx+\go(B^u)}=(x\mid
y)^{\gx+\go((B')^u)}(x_d-y_{e(B_{\ga})+1}).
\end{equation*}
On the other hand, if ${B_{\ga}}$ is barred, then
\begin{equation*}
c_{\gx+1,B}=c_{\gx+1,B'}(y_{e(B_{\ga})+1}-y_{f(B_{\ga})})\quad
\text{and}\quad
(x\mid y)^{\gx+\go(B^u)}=(x\mid y)^{\gx+\go((B')^u)}.
\end{equation*}
Thus,
\begin{align*}
\sum_{B\in\cB_{sub}(\gm)\atop \widetilde{B}=R} c_{\gx+1,B}&(x \mid y)^{\gx+\go(B^u)}\\
&= \sum_{B\in\cB_{sub}(\gm)\atop{\widetilde{B}=R\atop
{B_{\ga}=k}}} c_{\gx+1,B}(x \mid y)^{\gx+\go(B^u)} +
\sum_{B\in\cB_{sub}(\gm)\atop{\widetilde{B}=R\atop
{B_{\ga}=\ol{k}}}} c_{\gx+1,B}(x \mid
y)^{\gx+\go(B^u)}\\
&= \sum_{B\in\cB_{sub}(\gm)\atop{\widetilde{B}=R\atop
{B_{\ga}=k}}} c_{\gx+1,B'}(x \mid
y)^{\gx+\go((B')^u)}(x_{B_{\ga}}-y_{e(B_{\ga})+1})\\
&\hspace*{5em} +\sum_{B\in\cB_{sub}(\gm)\atop{\widetilde{B}=R\atop
B_{\ga}=\ol{k}}} c_{\gx+1,B'}(y_{e(B_{\ga})+1}-y_{f(B_{\ga})})(x
\mid y)^{\gx+\go((B')^u)}
\\
&= \sum_{B\in\cB_{sub}(\gm)\atop\widetilde{B}=R}
\left(c_{\gx+1,B'}(x \mid
y)^{\gx+\go((B')^u)}(x_{{B_{\ga}}}-y_{e(B_{\ga})+1})\right.\\
&\hspace*{5em}\left. +\quad
c_{\gx+1,B'}(y_{e(B_{\ga})+1}-y_{f(B_{\ga})})(x \mid
y)^{\gx+\go((B')^u)}\right)\\
&= \sum_{B\in\cB_{sub}(\gm)\atop\widetilde{B}=R} c_{\gx+1,B'}(x
\mid
y)^{\gx+\go((B')^u)}(x_{B_{\ga}}-y_{f(B_{\ga})})\\
%
%
&= \sum_{D\in\cB_{sub}(\gm)\atop\widetilde{D}=R'}
\left(c_{\gx+1,D}(x \mid
y)^{\gx+\go(D^u)}\right)(x_{B_{\ga}}-y_{f(B_{\ga})})\\
&=(x \mid y)^{\gx}(x \mid y)^{R'}(x_{B_{\ga}}-y_{f(B_{\ga})})\\
&=(x \mid y)^{\gx}(x \mid y)^{R}.
\end{align*}
\end{proof}

\begin{proof}[Proof of Lemma \ref{l.lra}]
\begin{align*}
a_{\gl+\gr}\xy s_{\gm}\xy &\stackrel{\tiny (a)}{=}\sum_{\gs\in
S_d}\sum_{R\in
\cR(\gm)}\sgn(\gs)(x\mid y)^{\gs(\gl+\gr)}(x\mid y)^R\\
%
%
&\stackrel{\tiny (b)}{=}\sum_{\gs\in S_d}\sum_{R\in
\cR(\gm)}\sum_{B\in \cB(\gm)\atop \widetilde{B}=R}
c_{\gs(\gl+\gr+1),B}\sgn(\gs)(x\mid y)^{\gs(\gl+\gr)+\go(B^u)}\\
&\stackrel{\tiny (c)}{=}\sum_{\gs\in S_d}\sum_{R\in
\cR(\gm)}\sum_{B\in \cB(\gm)\atop \widetilde{B}=R}
\sum_{H\in\cH(\gm)\atop\ol{H}=B}d_{\gs(\gl+\gr+1),H}\sgn(\gs)(x\mid
y)^{\gs(\gl+\gr)+\go(H^u)}\\
&=\sum_{\gs\in S_d}\sum_{H\in
\cH(\gm)}d_{\gs(\gl+\gr+1),H}\sgn(\gs)(x\mid
y)^{\gs(\gl+\gr)+\go(H^u)}\\
&\stackrel{\tiny (d)}{=}\sum_{\gs\in S_d}\sum_{H\in
\cH(\gm)}d_{\gs(\gl+\gr+1),\gs H}\sgn(\gs)(x\mid
y)^{\gs(\gl+\gr)+\go((\gs H)^u)}\\
&\stackrel{\tiny (e)}{=}\sum_{\gs\in S_d}\sum_{H\in
\cH(\gm)}d_{\gl+\gr+1,H}\sgn(\gs)(x\mid
y)^{\gs(\gl+\gr+\go(H^u))}\\
&=\sum_{\gs\in S_d}\sum_{B\in \cB(\gm)}
\sum_{H\in\cH(\gm)\atop\ol{H}=B}d_{\gl+\gr+1,H}\sgn(\gs)(x\mid
y)^{\gs(\gl+\gr+\go(H^u))}\\
&\stackrel{\tiny (c)}{=}\sum_{\gs\in S_d}\sum_{B\in \cB(\gm)}
c_{\gl+\gr+1,B}\sgn(\gs)(x\mid
y)^{\gs(\gl+\gr+\go(B^u))}\\
&\stackrel{\tiny (a)}{=}\sum_{B\in \cB(\gm)}
c_{\gl+\gr+1,B}a_{\gl+\gr+\go(B^u)}\xy
\stackrel{\tiny (f)}{=}\sum_{B\in \cB(\gm)}
c_{\gl,B}a_{\gl+\gr+\go(B^u)}\xy.
\end{align*}
Equality (a) follows from the definition of $a_\gm$, noting that
$\gs(\gl+\gr)+1=\gs(\gl+\gr+1)$; (b) follows from Lemma
\ref{l.pra_interm}, setting $S=R$ and $\gx=\gs(\gl+\gr)$; (c) from
(\ref{e.lr_coeff_hatted_coeff}), with $\gx=\gs(\gl+\gr)$; (e) from
(\ref{e.perm_properties}); and (f) from (\ref{e.star_nostar}). For
(d), we use the fact that for a fixed $\gs$ and arbitrary
decomposition $\gs=\gs_{i_1}\cdots\gs_{i_t}$, since each $s_{i_j}$
is an involution on $\cH(\gm)$, as $H$ runs over all elements of
$\cH(\gm)$, so does $\gs H$.
\end{proof}

\section{Proof of Lemma \ref{l.bad_guys_vanish}}\label{s.proof_lemma2}

By (\ref{e.lr_coeff_h_b}), Lemma \ref{l.bad_guys_vanish} is
equivalent to the following lemma, whose statement and proof
generalize arguments in \cite{Stem1}. For $H\in\cH(\gm)$ and $j$ a
nonnegative integer, define $H_{<j}$ to be the sub-hatted tableau
of $H$ consisting of the portion of $H$ lying in columns to the
right of $j$, and $H^u_{<j}=(H_{<j})^u$ (and similarly for
$H_{\leq j}$, $H_{>j}$, etc.).
\begin{lem}\label{l.bad_guys_vanish_hatted} Let $\gl\in\cP_n$. Then
\begin{equation}\label{e.bad_guys_vanish_hatted}
\sum d_{\gl+\gr+1,H} a_{\gl+\gr+\go(H^u)}\xy=0,
\end{equation}
the sum being over all $H\in \cH(\gm)$ for which
$\gl+\go(H^u_{\leq j})\not\in \cP_d$ for some $j$.
\end{lem}

\begin{proof} We call $H\in \cH(\gm)$ for which $\gl+\go(H^u_{\leq
j})\not\in \cP_d$ for some $j$ a {\it Bad Guy}. Let $H$ be a Bad
Guy, and let $j$ be minimal such that $\gl+\go(H^u_{\leq
j})\not\in \cP_d$. Having selected $j$, let $i$ be minimal such
that $(\gl+\go(H^u_{\leq j}))_{i}<(\gl+\go(H^u_{\leq j}))_{i+1}$.
Since $(\gl+\go(H^u_{\leq j-1}))_{i}\geq(\gl+\go(H^u_{\leq
j-1}))_{i+1}$ (by the minimality of $j$), we must have
$(\gl+\go(H^u_{\leq j-1}))_{i}=(\gl+\go(H^u_{\leq j-1}))_{i+1}$,
and column $j$ of $H$ must have an unhatted $i+1$ but not an
unhatted $i$. Thus
\begin{equation}\label{e.tech3}
(\gl+\gr+1+\go(H^u_{\leq j}))_{i}=(\gl+\gr+1+\go(H^u_{\leq
j}))_{i+1}.
\end{equation}

Define $H^*$ to be the reverse tableau of shape $\gm$ obtained
from $H$ by replacing $H_{>j}$ by $s_i(H_{>j})$. Notice first that
$H^*$ is still semistandard. Indeed, since $\gs_i$ applied to
$H_{>j}$ can only change the values of its entries from $i$ to
$i+1$ and visa-versa, the only possible violation of
semistandardness of $H^*$ would occur under the following
scenario: (a) $H$ has an entry $\fa$ of value $i$ in column $j$
(which has to be either an $\wh{i}$ or $\hat{i}$), (b) $H$ has an
entry $\fb$ of value $i$ immediately to the left of $\fa$, and (c)
$s_i$ applied to $H_{>j}$ changes the value of $\fb$ to $i+1$.
However, this scenario is impossible. If (a) and (b) both hold,
then since $H$ is semistandard, the entry of $H$ immediately below
$\fb$ must have value $i+1$ (we remark that the {\it reverse}
shape of the tableau is critical here). Therefore the entry in box
$\fb$ is not a free entry of $H_{>j}$, so $s_i$ does not change
its value, i.e., (c) is violated. Notice second that $H^*$ is
still a Bad Guy, since $H^*_{\leq j}=H_{\leq j}$. Thus $H\mapsto
H^*$ gives an involution on the set of Bad Guys of $H(\gm)$.

We define maps $b^*_l:H^l\to (H^*)^l$ and $b^*_r:H^r\to (H^*)^r$,
as follows. If $a\in (H_{\leq j})^l$, then define $b^*_l(a)=a$. If
$a\in (H_{> j})^l$, then during the construction of $H^*$, in the
process of applying $s_i$ to $H_{>j}$, $a$ is mapped to $b_l(a)\in
(H_{>j})^l$. This same element $b_l(a)$, regarded as an element of
$(H^*)^l$, is denoted by $b^*_l(a)$. The map $b^*_r$ is defined
analogously.

We wish to show that for $a\in H^l$,
\begin{equation}\label{e.tech1}
e_{\gl+\gr+1,H}(a)=e_{\gl+\gr+1,H^*}(b^*_l(a)),
\end{equation}
and for $a\in H^r$,
\begin{equation}\label{e.tech2}
f_H(a)=f_{H^*}(b^*_r(a)).
\end{equation}
For $a\in (H_{\leq j})^l$ or $a\in (H_{\leq j})^r$ , both
(\ref{e.tech1}) and (\ref{e.tech2}) are obvious.  The proof of
(\ref{e.tech2}) for $a\in (H_{> j})^r$ follows in much the same
manner as the proof of Lemma \ref{l.invol_properties}(v).

It remains to prove (\ref{e.tech1}) for $a\in (H_{>j})^l$. For
such $a$, by Lemma \ref{l.invol_properties}(iii),
\begin{equation}\label{e.tech5}
\gs_i\go(_{j<}H^u_{<c(a)})=\go(s_i(_{j<}H^u_{<c(b_l(a))}))=\go(_{j<}(H^*)^u_{<c(b^*_l(a))}),
\end{equation}
where $_{l<}H_{<m}:=H_{>l}\cap H_{<m}$. By (\ref{e.tech3}),
\begin{equation}\label{e.tech4}
\gs_i(\gl+\gr+1+\go(H^u_{\leq j} ))=\gl+\gr+1+\go(H^u_{\leq
j})=\gl+\gr+1+\go((H^*)^u_{\leq j}).
\end{equation}

\ni Thus
\begin{align*}
e_{\gl+\gr+1,H}(a)&=\left(\gl+\gr+1+\go(H^u_{<c(a)}))\right)_{|a|}\\
&=\left(\gs_i(\gl+\gr+1+\go(H^u_{<c(a)}))\right)_{\gs_i|a|}\\
&\stackrel{(a)}{=}\left(\gs_i(\gl+\gr+1+\go(H^u_{<c(a)}))\right)_{|b^*_l(a)|}\\
&=\left(\gs_i(\gl+\gr+1+\go(_{j<}H^u_{<c(a)})+\go(H^u_{\leq j}))\right)_{|b^*_l(a)|}\\
&=\left(\gs_i(\go(_{j<}H^u_{<c(a)}))+\gs_i(\gl+\gr+1+\go(H^u_{\leq j}))\right)_{|b^*_l(a)|}\\
&\stackrel{(b)}{=}\left(\go(_{j<}(H^*)^u_{<c(b^*_l(a))})+\gl+\gr+1+\go((H^*)^u_{\leq
j})\right)_{|b^*_l(a)|}\\
\end{align*}
\begin{align*}
&=\left(\gl+\gr+1+\go((H^*)^u_{<c(b^*_l(a))})\right)_{|b^*_l(a)|}\\
&=e_{\gl+\gr+1,H^*}(b^*_l(a)).
\end{align*}
Equality (a) follows from Lemma \ref{l.invol_properties}(i); (b)
follows from (\ref{e.tech5}) and (\ref{e.tech4}). This completes
the proofs of (\ref{e.tech1}) and (\ref{e.tech2}).

Now (\ref{e.tech1}) and (\ref{e.tech2}) imply
\begin{equation}\label{e.d_hs_equals_d_h}
\begin{split}
d_{\gl+\gr+1,H}&=\prod_{a\in
H^{l}}y_{e_{\gl+\gr+1,H}(a)}\prod_{a\in
H^{r}}(-y_{f_H(a)})\\
&=\prod_{a\in H^{l}}y_{e_{\gl+\gr+1,H^*}(b^*_l(a))}\prod_{a\in
H^{r}}(-y_{f_{H^*}(b^*_l(a))})\\
&=\prod_{a\in (H^*)^{l}}y_{e_{\gl+\gr+1,H^*}(a)}\prod_{a\in
(H^*)^{r}}(-y_{f_{H^*}(a)})
=d_{\gl+\gr+1,H^*}.
\end{split}
\end{equation}

\ni By $\gs_i\go(H^u_{>j} )=\go((H^*)^u_{>j})$ and
(\ref{e.tech4}), $\gs_i(\gl+\gr+\go(H^u))=(\gl+\gr+\go((H^*)^u))$;
thus
\begin{equation}\label{e.as_equal}a_{\gl+\gr+\go(H^u)}\xy=-a_{\gl+\gr+\go((H^*)^u)}\xy.
\end{equation}
By (\ref{e.d_hs_equals_d_h}) and (\ref{e.as_equal}), the
contributions to (\ref{e.bad_guys_vanish_hatted}) of two Bad Guys
paired under the involution $H\mapsto H^*$ are negatives, and thus
cancel.  If a Bad Guy is paired with itself under $H\mapsto H^*$,
then (\ref{e.as_equal}) implies that its contribution to
(\ref{e.bad_guys_vanish_hatted}) is 0.
\end{proof}

\section{Bijection with Knutson-Tao Puzzles}\label{s.bij_knutson-tao}

In this section we give a weight-preserving bijection between
positive equivariant Littlewood-Richardson skew tableaux and
Knutson-Tao puzzles; thus both combinatorial objects compute
identical expressions for the structure constants $c\lmn$ and
$C\lmn$, $\gl,\gm,\gn\in\cP_{d,n}$. We begin by reviewing the
construction of Knutson-Tao puzzles.

\subsection{Puzzles}

A \textbf{puzzle piece} is one of the eight figures shown in
Figure \ref{f.apic2_pp}, each of whose edges has length $1$ unit.
Each puzzle piece is either an equilateral triangles or a rhombus,
together with a fixed orientation, and a labelling of each edge
with either a 1 or 0. The rightmost puzzle piece in Figure
\ref{f.apic2_pp} is called an \textbf{equivariant puzzle piece};
we color it cyan (or light gray on a black and white printer).
\begin{figure}[!h]
\input{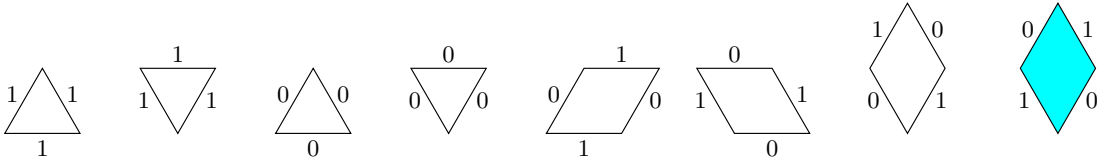}
\caption{\label{f.apic2_pp}The eight puzzle pieces}
\end{figure}

A \textbf{puzzle} $P$ is a partitioning of an equilateral triangle
of side length $n$ into puzzle pieces (see Figure
\ref{f.apic4_pex}). Implicit in this definition is that if two
puzzle pieces of $P$ share an edge, then both puzzle pieces must
have the same label on that edge.  The large equilateral triangle
forming the boundary of $P$ is called simply the \textbf{boundary}
of $P$ and denoted by $\bd P$. The northeast, northwest, and south
sides of the boundary are denoted by $\bd P_{\text{NE}}$, $\bd
P_{\text{NW}}$, and $\bd P_{\text{S}}$ respectively. One forms
three $n$-digit binary words by reading the labels along the three
sides of $\bd P$: the labels of $\bd P_{\text{NE}}$ are read from
top to bottom, the labels of $\bd P_{\text{NW}}$ from bottom to
top, and the labels of $\bd P_{\text{S}}$ from left to right. To
these three binary words we associate three partitions of
$\cP_{d,n}$ under the map $w\mapsto
(\eta_1,\ldots,\eta_d)\in\cP_{d,n}$, where $\eta_j$ is the number
of zeros of $w$ which lie to the right of the $j$-th one of $w$
from the left (for example, $0110001010\mapsto
(5,5,2,1)\in\cP_{4,10}$). Denote by $\cL\cP\lmnp$ the set of all
puzzles $P$ for which these three partitions are $\gl$, $\gm$, and
$\gn$, in that order.
\begin{figure}[!h]
\input{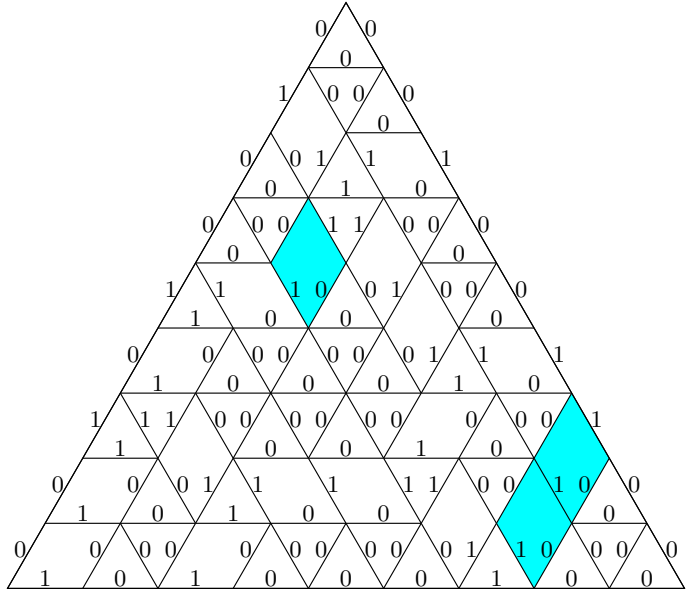}
\caption{\label{f.apic4_pex}A puzzle $P$, with $n=9$, $d=3$. The
$n$-digit binary words of the NE, NW, and S sides of the boundary
are $001001100$, $001010010$, and $101000100$ respectively. Thus
$P\in \cL\cP\lmnp$, where $\gl=(4,2,2)$, $\gm=(4,3,1)$, and
$\gn=(6,5,2)$.}
\end{figure}

For any equivariant puzzle piece of a puzzle $P$, draw two lines
from the center of the puzzle piece to $\bd P_{\text{S}}$: one
line $L_1$ parallel to $\bd P_{\text{NW}}$ and the other $L_2$
parallel to $\bd P_{\text{NE}}$ (see Figure \ref{f.apic7_pewt}).
The line segment $\bd P_{\text{S}}$ consists of $n$ edges of
puzzle pieces, which we number $1,2,\ldots,n$ from right to left.
The lines $L_1$ and $L_2$ cross $\bd P_{\text{S}}$ in the center
of two edges $e$ and $f$ respectively, where $e>f$. The
\textbf{factorial weight of the puzzle piece} is $y_e-y_f$, and
the \textbf{equivariant weight of the puzzle piece} is
$Y_{n+1-f}-Y_{n+1-e}$. Let $c_P$ denote the product of the
factorial weights of all the equivariant puzzle pieces of $P$ and
$C_P$ the product of the equivariant weights of all the
equivariant puzzle pieces of $P$. For example, in Figure
\ref{f.apic4_pex}, $c_P=(y_8-y_3)(y_3-y_2)(y_3-y_1)$ and
$C_P=(Y_7-Y_2)(Y_8-Y_7)(Y_9-Y_7)$.

\begin{figure}[!h]
\input{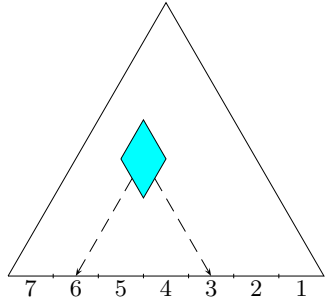}
\caption{\label{f.apic7_pewt} The equivariant puzzle piece has
factorial weight $y_6-y_3$ and equivariant weight $Y_5-Y_2$.}
\end{figure}
\vs

\begin{prop}\label{p.puzzles_tableaux_bij} There is a
weight preserving bijection $\gP:\cL\cP\lmnp\to\cL\cR\lmnp$.  By
weight-preserving, we mean that for $P\in\cL\cP\lmnp$, $c_P$ and
$c_{\gP(P)}$ are equal, and moreover are identical expressions;
and similarly for $C_P$ and $C_{\gP(P)}$.
\end{prop}

\begin{proof}
The bijection $\gP$, illustrated in Figure \ref{f.apic1_p_t},
generalizes Tao's `proof without words' of the bijection between
puzzles and tableaux in the nonequivariant setting \cite[Figure
11]{Va}. The large triangle in the center of Figure
\ref{f.apic1_p_t} represents a generic puzzle $P$.  The
equivariant Littlewood-Richardson skew tableau $\gP(P)$ is the
skew barred tableau formed by placing the Young tableau on the
top-right of Figure \ref{f.apic1_p_t} above and to the right of
the reverse barred tableau on the top left of the figure.

We list three properties of any $L\in\cL\cR\lmnp$:
\begin{itemize}
\item[(a)] $L|_\gm$ is column strict;

\item[(b)] the unbarred column word of $L$ is Yamanouchi; and

\item[(c)] $c_L>0$.
\end{itemize}

Let $P\in\cL\cP\lmnp$. For $i\in\{1,\ldots,d\}$ (where $d=4$ in
Figure \ref{f.apic1_p_t}), there is a path $P_i$ consisting of
1-triangles and rhombi which begins on $\bd P_{\text{NE}}$, moves
only west or southwest, and ends on $\bd P_{\text{S}}$
(see Figure \ref{f.apic1_p_tb}). Each path $P_i$ has
\textbf{segments} $P_{i,j}$ consisting of the consecutive rhombi
lying to the right of an upward pointing 1-triangle and to the
left of either a downward pointing 1 triangle or $\bd
P_{\text{NE}}$. We list three properties of $P$:
\begin{itemize}
\item[(a)'] for $i=2,\ldots,d$ and all $j$, the distance from the
leftmost edge of $P_{i,j}$ to $\bd P_{\text{NE}}$ is greater than
or equal to the distance from the leftmost edge of $P_{i-1,j}$ to
to $\bd P_{\text{NE}}$;

\item[(b)'] the interiors of the $P_i$ do not touch; and

\item[(c)'] the interiors of all equivariant puzzle pieces lie
above $\bd P_{\text{S}}$.
\end{itemize}

Given any $P\in\cL\cP\lmnp$, Figure \ref{f.apic1_p_t} shows how to
construct a skew barred tableau $\gP(P)$. Properties (a)', (b)',
and (c)' of $P$ imply properties (a), (b), and (c) of $\gP(P)$
respectively. Conversely, given any $L\in\cL\cR\lmnp$, Figure
\ref{f.apic1_p_t} shows how to construct a puzzle $\gP^{-1}(L)$.
Properties (a), (b), and (c) of $L$ ensure that the puzzle
$\gP^{-1}(L)$ can be constructed, and imply that it satisfies
(a)', (b)', and (c)' respectively. Uniqueness is clear.

To each equivariant puzzle piece of $P$ there corresponds a barred
entry of $\gP(P)$, and they both determine the same factor
$y_i-y_j$ of $c_P$ and $c_{\gP(P)}$ respectively. Therefore $\gP$
is weight preserving.
\end{proof}

\ni Using Corollary \ref{c.lr_coeff} and Proposition
\ref{p.puzzles_tableaux_bij}, we obtain a new proof of the
following theorem, which is due to Knutson and Tao \cite{Kn-Ta}.

\begin{thm}[Knutson-Tao]\label{t.knutson_tao}
$\displaystyle c_{\gl,\gm}^{\gn}=\sum_{P\in \cL\cP\lmnp}c_P$ and
$\displaystyle C_{\gl,\gm}^{\gn}=\sum_{P\in \cL\cP\lmnp}C_P$.
\end{thm}


\begin{figure}[!p]
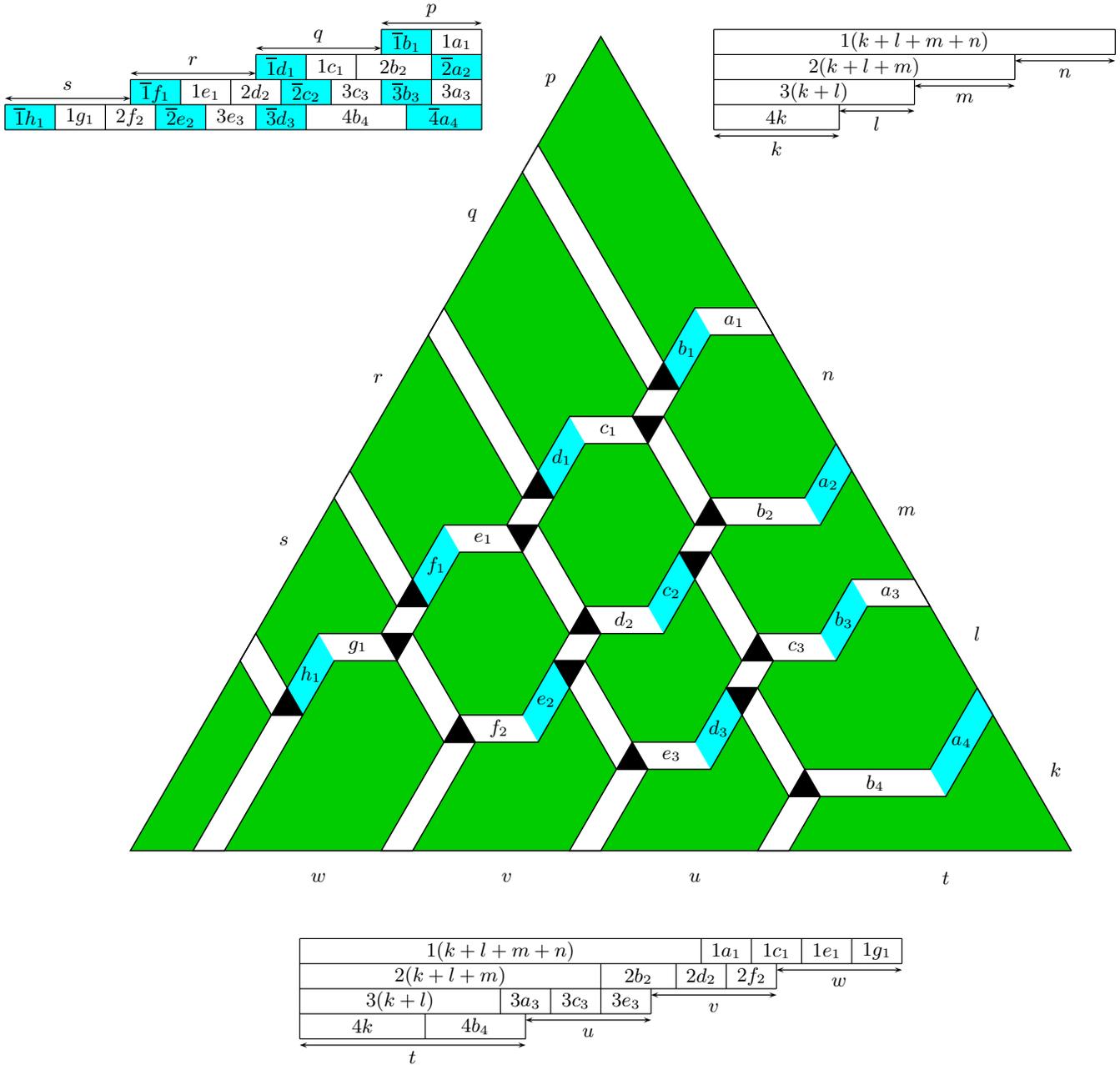

\include{apic1_p_t} \caption{\label{f.apic1_p_t} A generic puzzle
$P$ (center), and its associated positive equivariant
Littlewood-Richardson skew tableau $\gP(P)$ (top-right, top-left).
In $P$, black represents regions of 1 triangles, green (dark gray)
represents regions of 0 triangles, white represents regions of
non-equivariant rhombi, and cyan (light gray) represents regions
of equivariant rhombi.}
\end{figure}
\vs

\begin{figure}[!p]
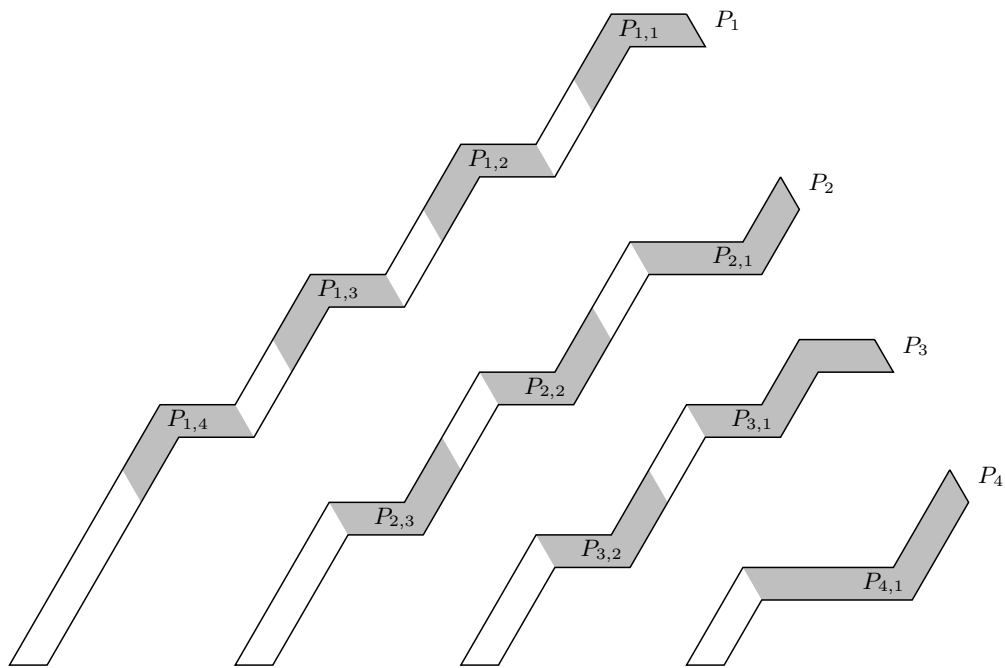

\include{apic1_p_tb} \caption{\label{f.apic1_p_tb} The paths
$P_i$, $i=1,\ldots,4$, of the puzzle $P$ of Figure
\ref{f.apic1_p_t}.  The segments $P_{i,j}$ of each path are
shaded. The segments may contain two types of puzzle pieces:
equivariant puzzle pieces and rhombi with horizontal 0-edges.}
\end{figure}
\vs

\begin{figure}[!ht]
\input{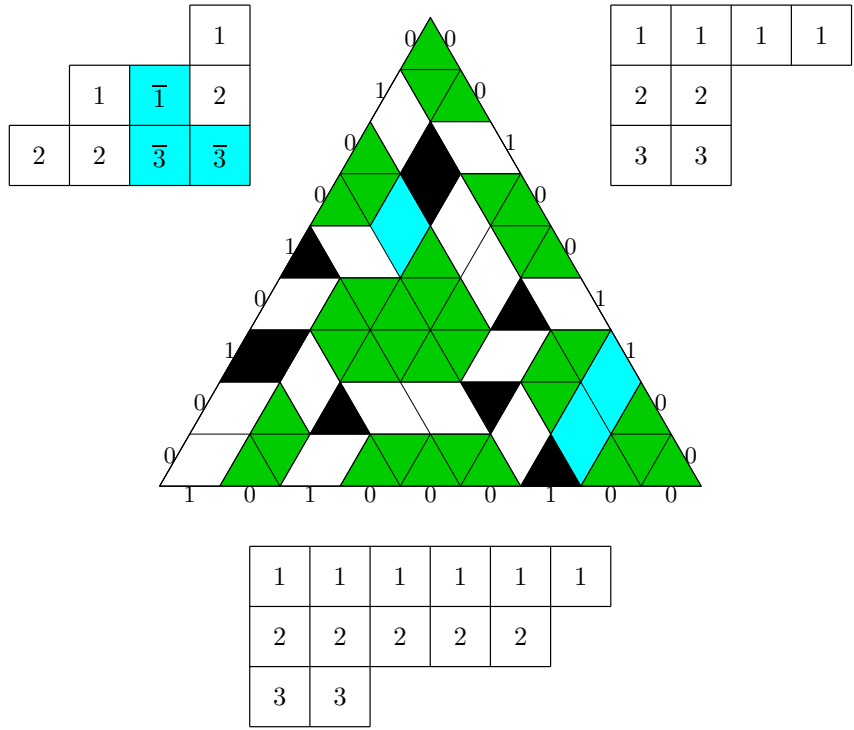}
\caption{\label{f.apic5_pexc} The puzzle $P$ of Figure
\ref{f.apic4_pex} is redrawn using the coloring scheme described
in Figure \ref{f.apic1_p_t}. All edge labels other than those on
the boundary of the puzzle are suppressed. $\gP(P)$ is also
shown.}
\end{figure}

\begin{figure}[!h]
$
\begin{array}{c@{\hspace{10em}}c}
\input{apic6a_plr}
&
\input{apic6b_plr}
\end{array}
$ \caption{\label{f.apic6_plr} All elements of $\cL\cP\lmnp$ and
$\cL\cR\lmnp$, for $d=2$, $n=4$, $\gl=(1,1)$, $\gm=(2,1)$, and
$\gn=(2,1)$.  We have
$c_{\gl,\gm}^{\gm}=(y_4-y_3)(y_2-y_1)+(y_3-y_1)(y_2-y_1)$ and
$C_{\gl,\gm}^{\gm}=(Y_2-Y_1)(Y_4-Y_3)+(Y_4-Y_2)(Y_4-Y_3)$.}
\end{figure}

\clearpage

\subsection{Trapezoid Puzzles}
We next extend $\gP$ to a bijection onto $\cL\cR\lmn$. To do so,
we increase the size of the domain of $\gP$ by defining
generalizations of puzzles, which we call {\it trapezoid puzzles}.
The extention of $\gP$, which we also denote by $\gP$, allows us
to view nonnegativity from the point of view of trapezoid puzzles
rather than equivariant Littlewood-Richardson skew tableaux.

Consider the isosceles trapezoid $T$ formed by placing an
equilateral triangle of side length $n$ on top of a rhombus of
side length $n$ (see Figure \ref{f.apic8_pn}). The boundary of
$T$, denoted by of $\bd T$, is divided into 5 parts: northeast,
northwest, east, west, and south (denoted by $\bd T_{\text{NE}}$,
$\bd T_{\text{NW}}$, $\bd T_{\text{E}}$, $\bd T_{\text{W}}$, and
$\bd T_{\text{S}}$). The northeast and northwest boundaries of $T$
are the northeast and northwest boundaries of the equilateral
triangle, and the east, west, and south boundaries of $T$ are the
east, west, and south boundaries of the rhombus. A
\textbf{trapezoid puzzle} is a partitioning of $T$ into puzzle
pieces in such a way that all labels of $\bd T_{\text{E}}$ and
$\bd T_{\text{W}}$ are 0's. Denote by $\cL\cP\lmn$ the set of all
trapezoid puzzles whose n-digit binary words read from $\bd
T_{\text{NE}}$, $\bd T_{\text{NW}}$, and $\bd T_{\text{S}}$
correspond to partitions $\gl$, $\gm$, and $\gn$ respectively.

Let $D$ denote the line segment forming the south border of the
triangle (and the north border of the rhombus). For any
equivariant puzzle piece of a trapezoid puzzle $P$, draw two lines
from the center of the puzzle piece to $D$: one line $L_1$
parallel to $\bd T_{\text{NW}}$ and the other $L_2$ parallel to
$\bd T_{\text{NE}}$. The lines $L_1$ and $L_2$ cross $D$ at $e-.5$
and $f-.5$ units from its right endpoint, respectively ($e$, $f$
are both integers). If the equivariant puzzle piece lies above
$D$, then $e>f$; if it lies below $D$, then $e<f$; if it is
bisected by $D$, then $e=f$. The \textbf{factorial weight of the
puzzle piece} is $y_e-y_f$, and the \textbf{equivariant weight of
the puzzle piece} is $Y_{n+1-f}-Y_{n+1-e}$. Let $c_P$ denote the
product of the factorial weights of all the equivariant puzzle
pieces of $P$ and $C_P$ the product of the equivariant weights of
all the equivariant puzzle pieces of $P$.

A puzzle can be viewed as a trapezoid puzzle all of whose
1-triangles lie above $D$. In this way $\cL\cP\lmnp$ may be viewed
as a subset of $\cL\cP\lmn$, and the inclusion is
weight-preserving. A diagram very similar to Figure
\ref{f.apic1_p_t}, but for trapezoid puzzles instead of puzzles,
can be used to prove

\begin{prop}\label{p.trap_puzzles_tableaux_bij}
The bijection $\gP:\cL\cP\lmnp\to\cL\cR\lmnp$ of Proposition
\ref{p.puzzles_tableaux_bij} extends to a weight preserving
bijection $\gP:\cL\cP\lmn\to\cL\cR\lmn$.
\end{prop}

One makes the following two observations: if $P\in
\cL\cP\lmn\setminus \cL\cP\lmnp$, i.e., if $P$ has a 1-triangle
lying below $D$, then (i) at least one equivariant puzzle piece
must be bisected by $D$, and (ii) stronger, at least one
equivariant puzzle piece corresponding to a bottom row element of
$\gP(P)$ must be bisected by $D$. Although we have not included a
diagram such as Figure \ref{f.apic1_p_t} for trapezoid puzzles,
these two statements can nevertheless be seen in Figure
\ref{f.apic1_p_t} itself, if one imagines placing the required
side length $n$ rhombus underneath the figure, and `stretching'
the `paths', expanding without breaking the `loops', so that the
1-triangles are pushed into the rhombus below. Statement (i)
implies that $c_P=C_P=0$. Combined with Proposition
\ref{p.trap_puzzles_tableaux_bij}, it gives a simpler proof of
Proposition \ref{p.nonnegativity}. Statement (ii) proves that
Lemma \ref{l.tech_pos_crit}(iii) implies Lemma
\ref{l.tech_pos_crit}(i).  That Lemma \ref{l.tech_pos_crit}(i)
implies Lemma \ref{l.tech_pos_crit}(ii) can also be seen easily by
considering trapezoid puzzles (or puzzles).

By Theorem \ref{t.knutson_tao} and Proposition
\ref{p.trap_puzzles_tableaux_bij}, we have

\begin{cor}
$\displaystyle c_{\gl,\gm}^{\gn}=\sum_{P\in \cL\cP\lmn}c_P$ and
$\displaystyle C_{\gl,\gm}^{\gn}=\sum_{P\in \cL\cP\lmn}C_P$.
\end{cor}

\vfill \eject

\begin{figure}[!ht]
\input{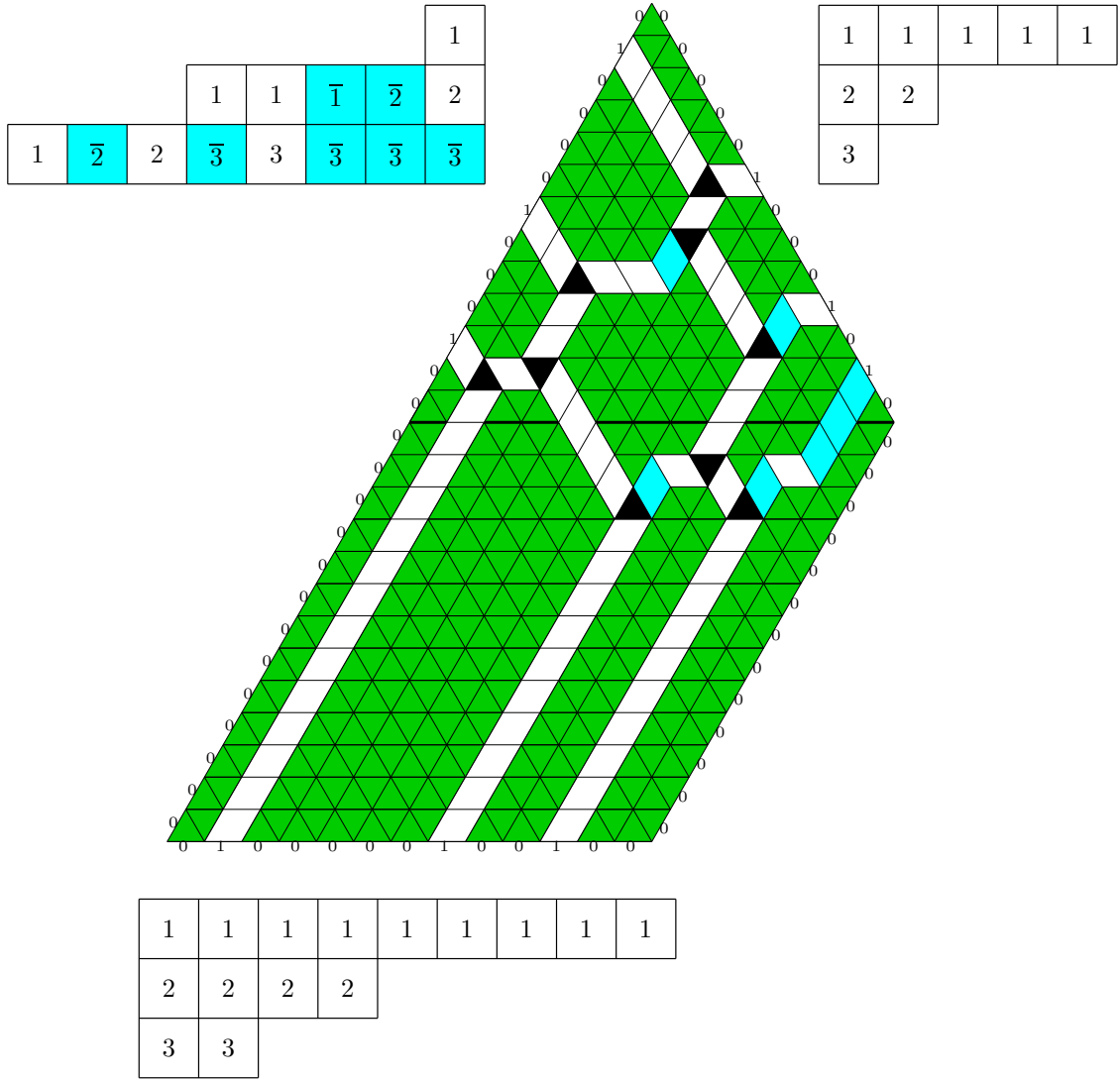}
\caption{\label{f.apic8_pn} A trapezoid puzzle $P$ (center), for
$d=3$, $n=13$, $\gl=(5,2,1)$, $\gm=(8,5,1)$, $\gn=(9,4,2)$; and
the corresponding equivariant Littlewood-Richardson skew tableau
$\gP(P)$ (top-right, top-left). The line $D$ separating the
triangle from the rhombus is darkened. The fact that 1-triangles
lie below $D$ implies that $c_P=C_P=0$. Indeed,
$c_{P}=(y_9-y_4)(y_5-y_2)(y_2-y_1)(y_2-y_2)(y_2-y_3)(y_3-y_5)(y_6-y_8)=0$,
and
$C_{P}=(Y_{10}-Y_5)(Y_{12}-Y_9)(Y_{13}-Y_{12})(Y_{12}-Y_{12})(Y_{11}-Y_{12})(Y_9-Y_{11})(Y_6-Y_8)=0$.}
\end{figure}

\eject

\providecommand{\bysame}{\leavevmode\hbox
to3em{\hrulefill}\thinspace}
\providecommand{\MR}{\relax\ifhmode\unskip\space\fi MR }
\providecommand{\MRhref}[2]{%
  \href{http://www.ams.org/mathscinet-getitem?mr=#1}{#2}
} \providecommand{\href}[2]{#2}


\end{document}